\documentclass[english]{article}
\usepackage[T1]{fontenc}
\usepackage[latin9]{inputenc}
\usepackage{color}
\usepackage{babel}
\usepackage{amsthm}
\usepackage{amsmath}
\usepackage{amssymb}
\usepackage{esint}
\usepackage[unicode=true]
 {hyperref}

\makeatletter
  \theoremstyle{definition}
  \newtheorem{defn}{\protect\definitionname}
  \theoremstyle{remark}
  \newtheorem{rem}{\protect\remarkname}
 \theoremstyle{definition}
  \newtheorem{example}{\protect\examplename}
\theoremstyle{plain}
\newtheorem{thm}{\protect\theoremname}
  \theoremstyle{plain}
  \newtheorem{lem}{\protect\lemmaname}
  \theoremstyle{plain}
  \newtheorem{prop}{\protect\propositionname}
  \theoremstyle{plain}
  \newtheorem{cor}{\protect\corollaryname}

\@ifundefined{definecolor}
 {\usepackage{color}}{}
\@ifundefined{definecolor}
 {\@ifundefined{definecolor}
 {\usepackage{color}}{}
}{}
\@ifundefined{definecolor}
 {\@ifundefined{definecolor}
 {\@ifundefined{definecolor}
 {\usepackage{color}}{}
}{}
}{}
\usepackage[all]{xy}

\newcommand{\xyR}[1]{
  \xydef@\xymatrixrowsep@{#1}}\newcommand{\xyC}[1]{
  \xydef@\xymatrixcolsep@{#1}}

\newdir{|>}{!/4.5pt/@{|}*:(1,-.2)@^{>}*:(1,+.2)@_{>}}

\let\myTOC\tableofcontents\renewcommand{\tableofcontents}{%
  \pdfbookmark[1]{\contentsname}{}
  \myTOC }

\def\LyX{\texorpdfstring{%
  L\kern-.1667em\lower.25em\hbox{Y}\kern-.125emX\@}
  {LyX}}

\makeatother

\usepackage{babel}

\makeatother

\usepackage{babel}

\makeatother

\usepackage{babel}

\date{}

\makeatother

  \providecommand{\definitionname}{Definition}
  \providecommand{\examplename}{Example}
  \providecommand{\lemmaname}{Lemma}
  \providecommand{\propositionname}{Proposition}
  \providecommand{\remarkname}{Remark}
\providecommand{\corollaryname}{Corollary}
\providecommand{\theoremname}{Theorem}

\begin{document}

\title{ON $J$-HOLOMORPHIC CURVES IN ALMOST COMPLEX MANIFOLDS WITH ASYMPTOTICALLY
CYLINDRICAL ENDS}

\author{ERKAO BAO}
\maketitle
\begin{abstract}
Symplectic Field Theory studies $J$-holomorphic curves in almost
complex manifolds with cylindrical ends. One natural generalization
is to replace ``cylindrical'' by ``asymptotically cylindrical''.
In this article, we generalize the asymptotic results about the behavior
of $J$-holomorphic curves near infinity in \cite{Hofer Weinstein conjecture,Finite energy plane,morse bott,compactness}
to the asymptotically cylindrical setting. We also sketch how these
asymptotic results allow the main compactness theorems of \cite{compactness}
to be extended to the asymptotically cylindrical case.
\end{abstract}

\author{%
\thanks{\noindent \protect\href{mailto: }{ }%
}}

\tableofcontents{}

\section{Introduction}

Introduced by Gromov in 1985, $J$-holomorphic curves have been studied
intensively in closed symplectic manifolds. In 1993 Hofer studied
the behaviors of $J$-holomorphic curves in symplectizations of contact
manifolds, which are noncompact. Shortly after that, Eliashberg, Givental
and Hofer invented the Symplectic Field Theory, which greatly helps
us understand symplectic manifolds and contact manifolds. In most
of previous literature, the almost complex structure $J$ allowed
is cylindrical near the ends of the noncompact symplectic manifolds.
Here cylindrical means that $J$ is independent of the radial direction.
In \cite{compactness} the notion of asymptotically cylindrical almost
complex structure was introduced, which is a natural generalization
of cylindrical almost complex structure. However, there is no corresponding
result proven for asymptotically cylindrical almost complex structure.
Intuitively, we expect similar results as in the cylindrical case.
However, the original proofs rely heavily on the cylindrical nature
of the almost complex structure, which prevents us from a direct generalization
to the asymptotically cylindrical case. In this paper, we give a modified
definition of asymptotically cylindrical almost complex structure
which includes an exponential decay condition that is satisfied in
all interesting examples, and prove some parallel analytical results
as in the cylindrical case. Based on these results we can compactify
the moduli space of $J$-holomorphic curves in almost complex manifolds
with asymptotically cylindrical ends by adding holomorphic buildings
introduced by \cite{compactness}.

This generalization is needed for application purposes. In many cases
the natural almost complex structure is only asymptotically cylindrical
(see Example \ref{exm: R(2n+2)} and Example \ref{exm:sub kahler}).
For instance, we can use the generalized results to prove the Gromov's
Monotonicity theorem with multiplicity (See \cite{Bao}). We also
take this chance to fill in some gaps in the literature.

In the asymptotically cylindrical case, the proofs of some theorems
are significantly different and more sophisticated than the proofs
in the cylindrical case (See the proofs of Proposition \ref{pro:gradient bound for finite energy curve},
Theorem \ref{thm:converge to reeb orbit} and Theorem \ref{thm:subsequence convergence to Reeb},
for example). The extra difficulties mainly come from the following
two facts: 1. the translations in the cylindrical almost complex manifold
are not $J$-holomorphic anymore; 2. the unmodified Hofer energy is
not positive when restricted to $J$-complex planes, and the modified
Hofer energy is not closed. Crucial uses of Gromov's Monotonicity
theorem are the main ingredients to overcome these extra difficulties.

In Section 2, we give the definition of asymptotically cylindrical
almost complex manifolds and the definition of Hofer energy of $J$-holomorphic
curves in this context. 

In Section 3, we give the proofs of the main results listed in Section
2. The proofs follow the schemes of \cite{Hofer Weinstein conjecture,Finite energy plane,Finite energy cylinders of small area,morse bott,compactness}. 

In Section 4, we give the definition of almost complex manifolds with
asymptotically cylindrical ends and the definition of Hofer energy
in this context. Finally we state and outline the proof of the compactness
result in this context.

\section{\label{sec:Asymptotically-cylindrical-almost}Asymptotically cylindrical
almost complex structures}

\subsection{\label{sub:Def of asymp cylind complex}Definition}

Let $V$ be a smooth closed oriented manifold of dimension $2n+1,$
and $J$ be a smooth almost complex structure in $W:=\mathbb{R}^{+}\times V.$
Assume that the orientation of $W$ determined by $J$ is the same
as the orientation coming from the standard orientation of $\mathbb{R}^{+}$
and the orientation of $V.$ Let $\mathbf{R}:=J\left(\frac{\partial}{\partial r}\right)$
be a smooth vector field on $W,$ and $\xi$ be a subbundle of the
tangent bundle $TW$ defined by $\xi_{(r,v)}=\left(0\times T_{v}V\right)\cap J\left(0\times T_{v}V\right)\subset T_{(r,v)}W$,
for $(r,v)\in W$. The tangent bundle $TW$ splits as $TW=\mathbb{R}(\frac{\partial}{\partial r})\oplus\mathbb{R}(\mathbf{R})\oplus\xi$. 

Define a 1-form $\lambda$ on $W$ by: $\lambda(\xi)=0$, $\lambda(\frac{\partial}{\partial r})=0$,
$\lambda\left(\mathbf{R}\right)=1$, and a 1-form $\sigma$ on $W$
by: $\sigma(\xi)=0$, $\sigma(\frac{\partial}{\partial r})=1$, $\sigma\left(\mathbf{R}\right)=0$.

We call a tensor on $W$ translationally invariant if it is independent
of the $r$-coordinate. Let $f_{s}:W\to W$ be the translation along
the $\mathbb{R}^{+}$-direction defined by $f_{s}(r,v):=(r+s,v).$
\begin{defn}
\label{def: asympt cylindrical}Under the above notations, $J$ is
called asymptotically cylindrical at positive infinity, if for all
$l\in\mathbb{Z}_{\geqq0},$ $J$ satisfies (AC1)-(AC5):\end{defn}
\begin{itemize}
\item (AC1) There exists a smooth translationally invariant almost complex
structure $J_{\infty}$ on $W$ and constants $K_{l}^{+},\delta_{l}>0$,
such that
\begin{equation}
\left\Vert \left.\nabla^{l}\left(J-J_{\infty}\right)\right|_{[r,+\infty)\times V}\right\Vert _{C^{0}}\leqq K_{l}^{+}e^{-\delta_{l}r}\label{eq:key}
\end{equation}
 for all $r\geq0$, where $\left\Vert \cdot\right\Vert _{C^{0}}$
is computed using a translationally invariant metric $g_{W}$ on $W$,
for example $g_{W}=dr^{2}+g_{V},$ and $\nabla$ is the corresponding
Levi-Civita connection. We further require that $K_{l}^{+}$ is sufficiently
small such that the $\omega$ defined in \eqref{eq:omega} satisfies
Requirements \ref{enu:-is-a} and \ref{enu:There-exist-constants}
in Section \ref{sub:Energy-of--holomorphic}. (See Remark \ref{rmk:1}.)
\item (AC2) $i(\mathbf{R}_{\infty})d\lambda_{\infty}=0,$ where $\mathbf{R}_{\infty}:=\underset{s\to\infty}{\lim}f_{s}^{*}\mathbf{R}$,
$\lambda_{\infty}:=\underset{s\to\infty}{\lim}f_{s}^{*}\lambda$,
and both limits exist by (AC1).
\item (AC3) $\mathbf{R}_{\infty}(r,v)=J_{\infty}\left(\frac{\partial}{\partial r}\right)\in0\times T_{v}V$. 
\end{itemize}
There exists a closed 2-form $\omega_{\infty}$ on $V$ such that 
\begin{itemize}
\item (AC4) $i(\mathbf{R}_{\infty})\omega_{\infty}=0.$
\item (AC5) $\omega_{\infty}(\cdot,J_{\infty}\cdot)$ is a metric on $\xi_{\infty},$
where $\xi_{\infty}=\underset{s\to\infty}{\lim}f_{s}^{*}\xi.$ \end{itemize}
\begin{rem}
\label{rmk:1}The definition we use is slightly different from the
one in \cite{compactness}. We require that $J$ converges to $J_{\infty}$
exponentially fast in condition (AC1). This is the accurate condition
to guarantee that the $J$-holomorphic curve converges to the periodic
orbits of $\mathbf{R}_{\infty}$ exponentially fast by the footnote
of formula (\ref{eq:zs+Mzt+Szout}). If we are only interested in
the behavior of a $J$-holomorphic curve near infinity, then the requirement
that $K_{l}^{+}$ is small can be achieved by restricting $W$ to
$r\geq r_{0}$ for some large $r_{0}.$ 
\end{rem}

We can restate the above conditions using the notion of hamiltonian
structure as in \cite{application of sft}. That the $2$-form $\omega_{\infty}$
has rank $2n$ says that $(V,\omega_{\infty})$ is a hamiltonian structure.
The conditions (AC3), $i(\mathbf{R}_{\infty})\omega_{\infty}=0=i(\mathbf{R}_{\infty})d\lambda_{\infty}$
and $\lambda_{\infty}(\mathbf{R}_{\infty})=1$ say that $(V,\omega_{\infty})$
is a stable hamiltonian structure. The condition $\xi_{\infty}=\ker\lambda_{\infty},$
that $J_{\infty}$ is an almost complex structure on $\xi_{\infty}$,
and that $J_{\infty}$ is compatible with $\omega_{\infty}$ ( by
(AC5) ) imply that $(\lambda_{\infty},J_{\infty})$ is a framing of
$(V,\omega_{\infty}).$ If in addition $\omega_{\infty}=d\lambda_{\infty},$
then we say $(V,\omega_{\infty})$ is of contact type. 

We call $(\lambda,J)$ defined as above an asymptotically cylindrical
framing of the stable hamiltonian structure $(V,\omega_{\infty})$.

Similarly, we can define the notion of $J$ being asymptotically cylindrical
on $\mathbb{R}^{-}\times V$ at $-\infty$. When we say $J$ is asymptotically
cylindrical, we choose $\omega_{\pm\infty}$ without mentioning.

The following definition is the case considered in \cite{Hofer Weinstein conjecture,Finite energy plane,Finite energy cylinders of small area,morse bott,compactness}.
\begin{defn}
\label{def:cylindrical almost complex}An almost complex structure
$J$ on $\mathbb{R}^{\pm}\times V$ is said to be a cylindrical almost
complex structure at $\pm\infty$, if $J$ is an asymptotically cylindrical
almost complex structure at $\pm\infty$ and $J$ is translationally
invariant near $\pm\infty$. 

An almost complex structure $J$ on $\mathbb{R}\times V$ is said
to be a cylindrical almost complex structure, if $J$ is asymptotically
cylindrical at both $\infty$ and $-\infty,$ and $J$ is translation
invariant.\end{defn}
\begin{example}
(Symplectization) Assume $(V,\xi)$ is a contact manifold with contact
$1$-form $\lambda$ and Reeb vector field $\mathbf{R}$, i.e. $\xi=\ker\lambda,$
$\lambda\wedge(d\lambda)^{n}\neq0$, $i_{\mathbf{R}}d\lambda=0,$
and $\lambda(\mathbf{R})=1.$ Let $\omega_{\infty}=d\lambda$ and
let $J_{\xi}$ be an almost complex structure in $\xi$ such that
it is compatible with $\omega_{\infty}|_{\xi}$, i.e. $d\lambda(\cdot,J_{\xi}\cdot)$
is a metric on $\xi.$ We extend the $J_{\xi}$ to $\mathbb{R}\times V$
by setting $J(\frac{\partial}{\partial r})=\mathbf{R}$. Then $J$
is a cylindrical almost complex structure, and in particular an asymptotically
cylindrical almost complex structure at $\pm\infty$.
\end{example}
Refer to \cite{compactness} for other interesting examples of cylindrical
almost complex structures.

\begin{example}
\label{exm: R(2n+2)}Assume $J$ is a smooth almost complex structure
on $\mathbb{R}^{2n+2}$ with $J(0)=J_{0}(0)$, where $J_{0}$ is the
standard complex structure on $\mathbb{R}^{2n+2}$. Consider $\mathbb{R}^{2n+2}\backslash\{0\}$
and pick a polar coordinate chart

\[
\varphi:\mathbb{R}^{-}\times S^{2n+1}\to\mathbb{R}^{2n+2}\backslash\{0\},
\]
\[
(r,\Theta)\mapsto e^{r}\Theta,
\]
where we view $S^{2n+1}$ as the unit sphere inside $\mathbb{R}^{2n+2}$.
Let $\lambda_{-\infty}$ be the standard contact form on $S^{2n+1}.$
Define the 2-form $\omega_{-\infty}$ on $\mathbb{R}^{-}\times S^{2n+1}$
by $\omega_{-\infty}=d\lambda_{-\infty}.$ Now it is clear that $J|_{\mathbb{R}^{-}\times S^{2n+1}}$
is an asymptotically cylindrical almost complex structure near $-\infty.$
\end{example}
By (AC1) and (AC3) we can see that $\mathbf{R}_{\infty}$ is a translationally
invariant vector field on $W$ and it is tangent to each level set
$\{r\}\times V$, so we can view $\mathbf{R}_{\infty}$ as a vector
field on $V$. Let $\phi^{t}$ be the flow of $\mathbf{R}_{\infty}$
on $V$, i.e. $\phi^{t}:V\to V$ satisfies $\frac{d}{dt}\phi^{t}=\mathbf{R}_{\infty}\circ\phi^{t}$.
Then we have 
\[
\frac{d}{dt}[(\phi^{t})^{*}\lambda_{\infty}]=(\phi^{t})^{*}(i(\mathbf{R}_{\infty})d\lambda_{\infty}+di(\mathbf{R}_{\infty})\lambda_{\infty})=0.
\]
Hence $\phi^{t}$ preserves $\lambda_{\infty}$ and thus also $\xi_{\infty}$.\textcolor{black}{{}
Similarly }$\phi^{t}$\textcolor{black}{{} preserves} $\omega_{\infty}$\textcolor{black}{.} 

Let's denote by $\mathcal{P}$ the set of periodic trajectories, counting
their multiples, of the vector field $\mathbf{R}_{\infty}$ restricting
to $V.$ Notice that any smooth family of periodic trajectories from
$\mathcal{P}$ has the same period by Stokes' Theorem. 
\begin{defn}
A $T$-periodic orbit $\gamma$ of $\mathbf{R}_{\infty}$ is called
non-degenerate, if $d\phi^{T}|_{\xi_{\infty}(\gamma(0))}$ does not
have $1$ as an eigenvalue, where $\phi^{t}$ is the flow of $\mathbf{R}_{\infty}.$
We say that $J$ is non-degenerate if all the periodic solutions of
$\mathbf{R}_{\infty}$ are non-degenerate. 
\end{defn}

A weaker requirement for $J$ than non-degenerate is Morse-Bott.
\begin{defn}
\label{def:Morse Bott}We say that $J$ is of the Morse-Bott type
if, for every $T>0$ the subset $N_{T}\subset V$ formed by the closed
trajectories from $\mathcal{P}$ of period $T$ is a smooth closed
submanifold of $V$, such that the rank of $\omega_{\infty}|_{N_{T}}$
is locally constant and $T_{p}N_{T}=\ker\left(d\phi^{T}-Id\right)_{p}$.
\end{defn}

\textbf{We always assume $J$ is of Morse-Bott type in this paper. }

\subsection{Energy of $J$-holomorphic curves\label{sub:Energy-of--holomorphic}}

Let $J$ be an asymptotically cylindrical almost complex structure
on $W:=\mathbb{R}^{+}\times V$. Let's denote the projections from
$TW=\mathbb{R}(\frac{\partial}{\partial r})\oplus\mathbb{R}(\mathbf{R})\oplus\xi$
to each subbundle by $\pi_{r}$,$\pi_{\mathbf{R}}$ and $\pi_{\xi}$.
It is convenient to introduce a $2$-form $\omega$ on $W$ by 
\begin{equation}
\omega(x,y)=\frac{1}{2}\left[\omega_{\infty}(\pi_{\xi}x,\pi_{\xi}y)+\omega_{\infty}(J\pi_{\xi}x,J\pi_{\xi}y)\right].\label{eq:omega}
\end{equation}
It is easy to check that $i\left(\frac{\partial}{\partial r}\right)\omega=0=i\left(\mathbf{R}\right)\omega.$
We assume that $K_{l}^{+}$ in (AC1) is sufficiently small for all
$l\in\mathbb{Z}_{\geqq0},$ such that $\omega$ satisfies:
\begin{enumerate}
\item \label{enu:-is-a}$\omega|_{\xi}(\cdot,J\cdot)$ is a metric on $\xi;$
and
\item \label{enu:There-exist-constants} There exist constants $\epsilon_{l},\delta_{l}>0$,
such that 
\[
\left\Vert \left.\left(\omega-\omega_{\infty}\right)\right|_{[r,+\infty)\times V}\right\Vert _{C^{l}}\leqq\epsilon_{l}e^{-\delta_{l}r}
\]
 for all $r\geqq0$.
\end{enumerate}
Let $(\Sigma,j)$ be a punctured Riemann surface (with or without
boundary) and $\tilde{u}=(a,u):(\Sigma,j)\to(W,J)$ be a $J$-holomorphic
curve, i.e. $T\tilde{u}\circ j=J(\tilde{u})\circ T\tilde{u}$. The
following definition is a modification of Hofer energy in cylindrical
almost complex structure case. The $\omega$-energy and $\lambda$-energy
are defined as follows respectively 
\[
E_{\omega}(\tilde{u})=\intop_{\Sigma}\tilde{u}^{*}\omega,
\]
\[
E_{\lambda}(\tilde{u})=\underset{\phi\in\mathcal{C}}{sup}\intop_{\Sigma}\tilde{u}^{*}(\phi(r)\sigma\wedge\lambda),
\]
where $\mathcal{C}=\{\phi\in C_{c}^{\infty}(\mathbb{R},[0,1])|\intop_{-\infty}^{+\infty}\phi(x)dx=1\}$%
\footnote{In \cite{compactness}, the set $\mathcal{C}$ is given by $\mathcal{C}=\{\phi\in C_{c}^{\infty}(\mathbb{R},\mathbb{R}^{+})|\intop_{-\infty}^{+\infty}\phi(x)dx=1\}$.
It is easier to get uniform energy bounds using the modified definition
in the case when the almost complex structure is only asymptotically
cylindrical.%
}, and $\lambda,\sigma$ are defined as in the beginning of subsection
\ref{sub:Def of asymp cylind complex}. Let's define the energy of
$\tilde{u}$ by 
\[
E(\tilde{u})=E_{\omega}(\tilde{u})+E_{\lambda}(\tilde{u}).
\]

Equip $\mathbb{R}^{+}\times S^{1}$ with the standard complex structure
and coordinate $(s,t)$, and consider a $J$-holomorphic map $\tilde{u}=(a,u):\mathbb{R}^{+}\times S^{1}\to W$.
Here we view $S^{1}$ as $\mathbb{R}/\mathbb{Z}$. Notice

\begin{align}
\tilde{u}^{*}\omega & =\omega(\pi_{\xi}\tilde{u}_{s},J(\tilde{u})\pi_{\xi}\tilde{u}_{s})ds\wedge dt,\label{eq:w-energy}
\end{align}
\begin{align}
\tilde{u}^{*}(\phi(r)\sigma\wedge\lambda) & =\phi(a)\left[\sigma(\tilde{u}_{s})^{2}+\lambda(\tilde{u}_{s})^{2}\right]ds\wedge dt.\label{eq:r-energy}
\end{align}
Thus, we have $E_{\omega}(\tilde{u})\geqq0$ and $E_{\lambda}(\tilde{u})\geqq0$.

\subsection{Main Results}

The following two theorems tell us the behaviors of $J$-holomorphic
curves near infinity.
\begin{thm}
\label{thm:converge to reeb orbit}Suppose that $J$ is an asymptotically
cylindrical almost complex structure on $\mathbb{R}^{\pm}\times V$
at $\pm\infty$. Suppose that $J$ is of the Morse-Bott type. Let
$\tilde{u}=(a,u):\mathbb{R}^{\pm}\times\mathbb{R}/\mathbb{Z}\to\mathbb{R}^{\pm}\times V$
be a finite energy $J$-holomorphic curve. Suppose that the image
of $\tilde{u}$ is unbounded in $\mathbb{R}^{\pm}\times V$. Then
there exists a periodic orbit $\gamma$ of $\mathbf{R}_{\infty}$
of period $|T|$ with $T\neq0$, such that

\[
\underset{s\to\pm\infty}{\lim}u(s,t)=\gamma(Tt)
\]
\[
\underset{s\to\pm\infty}{\lim}\frac{a(s,t)}{s}=T
\]
 in $C^{\infty}(S^{1})$. 
\end{thm}
The above theorem tells us that when $|s|$ is large enough $u(s,t)$
lies inside a small neighborhood of $\gamma$. We will construct a
coordinate chart for such a neighborhood $U\subset S^{1}\times\mathbb{R}^{2n}\to V$,
and then we can view the map $\tilde{u}$ as 
\[
\tilde{u}(s,t)=(a(s,t),\vartheta(s,t),z(s,t))\in\mathbb{R}^{\pm}\times\mathbb{R}\times\mathbb{R}^{2n},
\]
 where $\vartheta$ is the coordinate of the universal cover of $S^{1}=\mathbb{R}/\mathbb{Z}.$
\begin{thm}
\label{thm: exponential convergence} Under the same assumption as
in Theorem \ref{thm:converge to reeb orbit}, there exist constants
$M_{\beta},d_{\beta},a_{0},\vartheta_{0},s_{0}>0$ such that

\begin{align*}
|D^{\beta}\{a(s,t)-Ts-a_{0}\}| & \leqq M_{\beta}e^{\mp d_{\beta}s},\\
|D^{\beta}\{\vartheta(s,t)-Tt-\vartheta_{0}\}| & \leqq M_{\beta}e^{\mp d_{\beta}s},\\
|D^{\beta}z(s,t)| & \leqq M_{\beta}e^{\mp d_{\beta}s},
\end{align*}
for all $s>s_{0}$ and $\beta=(\beta_{1},\beta_{2})\in\mathbb{Z}_{\geqq0}\times\mathbb{Z}_{\geqq0}$.
\end{thm}

\section{Proof of main results}

The proofs for $\mathbb{R}^{+}\times V$ and $\mathbb{R}^{-}\times V$
are almost the same, so we will focus on the $\mathbb{R}^{+}\times V$
case. The proof is done in three steps. The first step is to show
that the gradient of a finite Hofer energy $J$-holomorphic curve
$\tilde{u}=(a,u)$ is bounded. The second step is to show ``subsequence
convergence'', briefly, given a sequence of numbers $R_{k}$ converging
to infinity, we want to show that there exists a subsequence $R_{k_{n}}$,
such that $u(R_{k_{n}},t)$ converges to a periodic solution of the
vector field $\mathbf{R}_{\infty}$. The third step is to get some
exponential decay estimate and then prove Theorem \ref{thm:converge to reeb orbit}
and Theorem \ref{thm: exponential convergence}.

\subsection{Gradient bounds}

We cite the following two lemmata for later use.
\begin{lem}
\label{Lem: Hofer's point set topology}\cite{Hofer Weinstein conjecture}
Let $(X,d)$ be a metric space. Equivalent are

(1) $(X,d)$ is complete.

(2) For every continuous map $\phi:X\to[0,+\infty)$ and given $x\in X,$
$\varepsilon>0$ there exist $x'\in X,$ $\varepsilon'>0$ such that\end{lem}
\begin{itemize}
\item $\varepsilon'\leqq\varepsilon,\:\phi(x')\varepsilon'\geqq\phi(x)\varepsilon$,
\item $d(x,x')\leqq2\varepsilon,$
\item $2\phi(x')\geqq\phi(y)$ for all $y\in X$ with $d(y,x')\leqq\varepsilon'$. 
\end{itemize}

Let $J$ be an asymptotically cylindrical almost complex structure
on $W=\mathbb{R}^{+}\times V$ at $\infty$, let $\tilde{u}=(a,u)$
be a $J$-holomorphic map from $B(0,R)$ to $W,$ where $B(z_{0},R):=z=\left\{ s+\sqrt{-1}t\in\mathbb{C}:\left|z-z_{0}\right|<R\right\} $,
denote

\begin{equation}
\left\Vert \nabla\tilde{u}\right\Vert :=\underset{(s,t)\in B(0,R)}{\sup}\left|\nabla\tilde{u}(s,t)\right|\label{eq:definition of gradient}
\end{equation}
 and 
\[
\left\Vert \tilde{u}\right\Vert _{C^{k}(B(0,R),W)}:=\underset{x\in B(0,R)}{\sup}\sum_{|l|=0}^{k}|\nabla^{l}\tilde{u}(x)|,
\]
where the norm $|\cdot|$ is computed with respect to the standard
metric $ds^{2}+dt^{2}$ on $B(z_{0},R)$ and a translationally invariant
metric $g_{W}$ on $W$, for example $g_{W}=g_{V}+dr^{2},$ and $\nabla$
is the the Levi-Civita connection with respect to $g_{W}$ on $W.$
The following lemma says that the gradient bound implies $C^{\infty}$
bound.
\begin{lem}
(Gromov-Schwarz)\label{lem:Gromov-Schwarz} Fix $0<\varepsilon<1$
and $k\in\mathbb{N}$, if $\left\Vert \nabla\tilde{u}\right\Vert <C'<+\infty,$
then there exists $C(k,C')>0$ such that 
\[
\left\Vert \tilde{u}\right\Vert _{C^{k}(B(0,R-\varepsilon),W)}\leqq C(k,C'),
\]
where $C(k,C')$ does not depend on $\tilde{u}$. \end{lem}
\begin{proof}
This is a standard result. Using the gradient bound of $\tilde{u}$,
we can find uniform coordinate charts both in domain and in target,
then we can apply Proposition 2.36 in \cite{Audin and lafon}. 
\end{proof}
The following proposition is one of the key steps in \cite{Hofer Weinstein conjecture}
whose proof reveals the relation between $\omega$ energy and trajectory
of $\mathbf{R}_{\infty}.$
\begin{prop}
\label{pro:Hofer's lemma}\cite{Hofer Weinstein conjecture} Suppose
$J$ is a cylindrical almost complex structure on $\mathbb{R}\times V$
and $\tilde{u}=(a,u):\mathbb{C}\to\mathbb{R}\times V$ is a finite
Hofer energy $J$-holomorphic plane (i.e. $E(\tilde{u})=E_{\lambda}(\tilde{u})+E_{\omega}(\tilde{u})<+\infty$).
If $E_{\omega}(\tilde{u})=0$ and $\left\Vert \nabla\tilde{u}\right\Vert \leqq C$
for some $C>0$, then $\tilde{u}$ is constant. \end{prop}
\begin{proof}
Suppose $\tilde{u}$ is not constant. By (\ref{eq:w-energy}), $\pi_{\xi}\tilde{u}_{s}=0=\pi_{\xi}\tilde{u}_{t}$.
Hence $\pi_{\xi}\circ T\tilde{u}$ is the zero section of $\tilde{u}^{*}\xi\to\mathbb{C}$.
Therefore we have $u(s,t)=x\circ f(s,t)$, where $x:\mathbb{R}\to V$
satisfies $\dot{x}=\mathbf{R}(x)$ and $f:\mathbb{C}\to\mathbb{R}$
is a smooth function. Consequently, $f_{s}=-a_{t}$; $f_{t}=a_{s}$.
Hence $\Phi:=f+ia$ is a holomorphic function on $\mathbb{C}$. Since
$\left\Vert \nabla\tilde{u}\right\Vert $ is bounded, $\left\Vert \nabla\Phi\right\Vert $
is bounded; thus $\Phi$ is a linear function. By (\ref{eq:r-energy})
\begin{align*}
E_{\lambda}(\tilde{u}) & =\underset{\phi\in\mathcal{C}}{\sup}\intop_{\mathbb{C}}\phi(a)(a_{s}^{2}+a_{t}^{2})ds\wedge dt=+\infty,
\end{align*}
 via a linear change of variables.
\end{proof}
The proposition below generalizes Proposition 27 in \cite{Hofer Weinstein conjecture}
to the asymptotically cylindrical case. 
\begin{prop}
\label{pro:gradient bound for finite energy curve} If $J$ is an
asymptotically cylindrical almost complex structure on $W=\mathbb{R}^{+}\times V$
at $\infty$, and $\tilde{u}$ is a $J$-holomorphic map from $\mathbb{C}$
to $W$ satisfying $E(\tilde{u})<+\infty$, then we get $\left\Vert \nabla\tilde{u}\right\Vert <+\infty$.\end{prop}
\begin{proof}
Suppose to the contrary, then there exist a sequence of points $z_{k}\in\mathbb{C}$,
satisfying, $|z_{k}|\to\infty$, $R_{k}:=\left\Vert \nabla\tilde{u}(z_{k})\right\Vert \to\infty$,
as $k\to\infty$. By Lemma \ref{Lem: Hofer's point set topology},
we can modify $z_{k}$ such that there exist a sequence of $\varepsilon_{k}>0$
satisfying: $\varepsilon_{k}\to0,$ $\varepsilon_{k}R_{k}\to+\infty,$
and $\left|\nabla\tilde{u}(z)\right|\leqq2R_{k}$ for $z\in B(z_{k},\varepsilon_{k})$.
Now there are two cases.

\textbf{Case1}. $\{a(z_{k})\}_{k\in\mathbb{Z}}$ is unbounded.

Then there exist a subsequence of $z_{k}$, still denoted by $z_{k}$,
such that $a(z_{k})\to+\infty$ or $a(z_{k})\to-\infty$. WLOG, let's
assume $a(z_{k})\to+\infty$. Pick a further subsequence of $z_{k}$,
such that $a(z_{k})\geqq2^{k+2}$. Let $\varepsilon_{k}^{'}:=\min\left\{ \varepsilon_{k},\frac{2^{k}}{R_{k}}\right\} $,
then we have $\varepsilon_{k}^{'}\to0$, $\varepsilon_{k}^{'}R_{k}\to+\infty$,
and $|a(z)-a(z_{k})|\leqq2\varepsilon_{k}^{'}R_{k}\leqq2\cdot\frac{2^{k}}{R_{k}}\cdot R_{k}=2^{k+1}$,
for $|z-z_{k}|\leqq\varepsilon_{k}^{'}$. Thus, $a(z)\geqq a(z_{k})-2^{k+1}\geqq2^{k+2}-2^{k+1}=2^{k+1}$,
for $|z-z_{k}|\leqq\varepsilon_{k}^{'}$.

Since $\tilde{u}$ is $J$-holomorphic, we have

\begin{equation}
J(\tilde{u})\circ T\tilde{u}=T\tilde{u}\circ i.
\end{equation}
 Thus
\begin{equation}
J_{\infty}(\tilde{u})\circ T\tilde{u}=T\tilde{u}\circ i+(J_{\infty}-J)(\tilde{u})\circ T\tilde{u}.\label{eq:1}
\end{equation}
 By (AC1), we have%
\footnote{Actually, to prove Proposition \ref{pro:gradient bound for finite energy curve},
Proposition \ref{pro:gradient bound for v holomorphic cylinder} and
Theorem \ref{thm:subsequence convergence to Reeb} we only need $f_{s}^{*}J\to J_{\infty}$
in $C_{loc}^{1}$, as $s\to\infty$. We need the stronger condition
(AC1) to prove exponential decay in \ref{sub:Exponential-decay} and
thus the main theorems.%
} 
\[
\underset{z\in B(z_{k},\varepsilon_{k}^{'})}{\sup}\left\Vert (J_{\infty}-J)(\tilde{u}(z))\right\Vert \to0,
\]
 as $k\to+\infty$. 

Define maps $\tilde{u}_{k}(z)=(a(z_{k}+z/R_{k})-a(z_{k}),u(z_{k}+z/R_{k}))$
from $\mathbb{C}$ to $\mathbb{R}\times V$. For any $R'>0$, when
$k$ is large, $\left\Vert \nabla\tilde{u}_{k}(z)\right\Vert \leqq2$
for $z\in B(0,R')$. By Lemma \ref{lem:Gromov-Schwarz}, for any $n\in\mathbb{Z}_{\geqq0},$
there exists $C(n,R')$ satisfying 
\begin{equation}
\left\Vert \tilde{u}_{k}\right\Vert _{C^{n}(B(0,R'-1),W)}\leqq C(n,R').\label{eq:ubar}
\end{equation}
 We also have

\begin{equation}
\left|\nabla\tilde{u}_{k}(0)\right|=1
\end{equation}
\begin{equation}
\left|\nabla\tilde{u}_{k}(z)\right|\leqq2\mbox{ \ensuremath{\mbox{ for all }}}|z|\leqq\varepsilon_{k}^{'}R_{k}.
\end{equation}
 We apply Ascoli-Arzela theorem to get a subsequence, still called
$\tilde{u}_{k}$, satisfying $\tilde{u}_{k}\to\tilde{u}_{\infty}$
in $C_{loc}^{\infty}$, as $k\to\infty$. Here $\tilde{u}_{\infty}:\mathbb{C}\to\mathbb{R}\times V$
is a $J_{\infty}$-holomorphic map satisfying 
\[
\begin{array}{cc}
\left|\nabla\tilde{u}_{\infty}(0)\right|=1\qquad & \left\Vert \nabla\tilde{u}_{\infty}\right\Vert \leqq2.\end{array}
\]
 Indeed, $\tilde{u}_{k}$ satisfies

\begin{equation}
J_{\infty}(\tilde{u}_{k})T\tilde{u}_{k}=T\tilde{u}_{k}i+o_{k},
\end{equation}
 where $\left\Vert o_{k}\right\Vert _{C^{0}\left(B(0,\varepsilon_{k}^{'}R_{k})\right)}\to0$
as $k\to\infty.$ Therefore, $\tilde{u}_{\infty}$ is $J_{\infty}$-holomorphic. 

Now let's look at its energy. 

\begin{align}
\intop_{B(0,R')}\hspace{-1em}\tilde{u}_{k}^{*}\omega_{\infty} & =\intop_{B(z_{k},R'/R_{k})}\hspace{-1em}\tilde{u}^{*}\omega+\intop_{B(z_{k},R'/R_{k})}\hspace{-1em}\tilde{u}^{*}(\omega-\omega_{\infty}).
\end{align}
 From $E(\tilde{u})<+\infty$ we see $\intop_{B(z_{k},R'/R_{k})}\tilde{u}^{*}\omega\to0,$
as $k\to+\infty$. While,

\begin{align*}
\left|\intop_{B(z_{k},R'/R_{k})}\hspace{-2em}\tilde{u}^{*}(\omega_{\infty}-\omega)\right| & \leqq\intop_{B(z_{k},R'/R_{k})}\hspace{-2em}(2R_{k})^{2}\left|(\omega_{\infty}-\omega)\left(\frac{\tilde{u}_{s}}{2R_{k}},\frac{\tilde{u}_{t}}{2R_{k}}\right)\right|ds\wedge dt\\
 & \leqq\pi\left(\frac{R'}{R_{k}}\right)(2R_{k})^{2}c_{k}\to0,
\end{align*}
 where $c_{k}:=\underset{z\in B(z_{k},\varepsilon_{k}^{'})}{\sup}\left|(\omega_{\infty}-\omega)\left(\frac{\tilde{u}_{s}}{2R_{k}},\frac{\tilde{u}_{t}}{2R_{k}}\right)\right|,$
and by (AC4) $c_{k}\to0$ as $k\to\infty.$ Therefore, 
\[
E_{\omega_{\infty}}(\tilde{u}_{\infty})=\intop_{\mathbb{C}}\tilde{u}_{\infty}^{*}\omega_{\infty}=0.
\]
Moreover, we have $E_{\lambda_{\infty}}(\tilde{u}_{\infty})<+\infty$.
Indeed, given $\phi\in\mathcal{C}$, denote $\phi_{k}(r):=\phi(r-a(z_{k}))\in\mathcal{C}$.
Then we have 
\begin{align}
 & \left|\intop_{B(0,R')}\hspace{-1em}\tilde{u}_{k}^{*}(\phi(r)dr\wedge\lambda_{\infty})\right|\nonumber \\
\leqq & \left|\intop_{B(z_{k},R'/R_{k})}\hspace{-2em}\phi_{k}(a)\tilde{u}^{*}(\sigma\wedge\lambda)\right|+\left|\intop_{B(z_{k},R'/R_{k})}\hspace{-2em}\phi_{k}(a)\tilde{u}^{*}(dr\wedge\lambda_{\infty}-\sigma\wedge\lambda)\right|.\label{eq:7}
\end{align}
While,

\begin{equation}
\left|\;\intop_{B(z_{k},R'/R_{k})}\hspace{-1.5em}\phi_{k}(a)\tilde{u}^{*}(\sigma\wedge\lambda)\right|\leqq\left|\;\intop_{\mathbb{C}}\phi_{k}(a)\tilde{u}^{*}(\sigma\wedge\lambda)\right|\leqq E_{\lambda}(\tilde{u})\label{eq:8}
\end{equation}
and

\begin{align}
 & \left|\intop_{B(z_{k},R'/R_{k})}\hspace{-1.5em}\phi_{k}(a)\tilde{u}^{*}(dr\wedge\lambda_{\infty}-\sigma\wedge\lambda)\right|\nonumber \\
\leqq & \intop_{B(z_{k},R'/R_{k})}\hspace{-1.5em}\phi_{k}(a)(2R_{k})^{2}\left|(dr\wedge\lambda_{\infty}-\sigma\wedge\lambda)\left(\frac{\tilde{u}_{s}}{2R_{k}},\frac{\tilde{u}_{t}}{2R_{k}}\right)\right|ds\wedge dt\nonumber \\
\leqq & \left(\underset{x\in\mathbb{R}}{\sup}\phi(x)\right)(2R_{k})^{2}r_{k}\pi\left(\frac{R'}{R_{k}}\right)^{2}\to0,\label{eq:10}
\end{align}
where $r_{k}:=\underset{z\in B(z_{k},R'/R_{k})}{\sup}\left|(dr\wedge\lambda_{\infty}-\sigma\wedge\lambda)\left(\frac{\tilde{u}_{s}}{2R_{k}},\frac{\tilde{u}_{t}}{2R_{k}}\right)\right|\to0$
as $k\to\infty.$ Combining (\ref{eq:7}),(\ref{eq:8}),(\ref{eq:10}),
we get: given $R'>0$ and $\phi\in\mathcal{C}$, there exists a constant
$K$  such that for all $k>K,$

\[
\left|\intop_{B(0,R')}\tilde{u}_{k}^{*}(\phi(r)dr\wedge\lambda_{\infty})\right|\leqq E_{\lambda}(\tilde{u})+1.
\]
 Therefore, $E_{\lambda_{\infty}}(\tilde{u}_{\infty})\leqq E_{\lambda}(\tilde{u})+1$.
Altogether, we get a $J_{\infty}$-holomorphic map $\tilde{u}_{\infty}:\mathbb{C}\to W$
satisfying

\[
\begin{array}{cc}
\;\:\:\left\Vert \nabla\tilde{u}_{\infty}\right\Vert \leqq2\qquad & \left|\nabla\tilde{u}_{\infty}(0)\right|=1\\
E_{\omega_{\infty}}(\tilde{u}_{\infty})=0 & \qquad E(\tilde{u}_{\infty})<+\infty.
\end{array}
\]

By Proposition \ref{pro:Hofer's lemma}, we get a contradiction, which
finishes the proof for Case 1.

\textbf{Case 2}: $\{a(z_{k})\}_{k\in\mathbb{Z}}$ is bounded.

Now let us define $\tilde{u}_{k}$ differently from Case 1 by: 
\[
\tilde{u}_{k}(z):=\tilde{u}\circ l_{k}=(a(z_{k}+z/R_{k}),u(z_{k}+z/R_{k})),
\]
then $\tilde{u}_{k}$ satisfies 
\[
\begin{array}{cc}
\left|\nabla\tilde{u}_{k}(z)\right|\leqq2\;\mbox{for }z\in B(0,\varepsilon_{k}R_{k});\\
\{\tilde{u}_{k}(0)\}_{k\in\mathbb{Z}^{+}}\mbox{ is bounded; }\qquad\quad & \left|\nabla\tilde{u}(0)\right|=1.
\end{array}
\]
Similarly as in Case 1, by applying Ascoli-Arzela theorem we get a
subsequence still called $\tilde{u}_{k}$ converging to $\tilde{u}_{\infty}=(a_{\infty,}u_{\infty}):\mathbb{C}\to W$
in the $C_{loc}^{\infty}$ sense. Here $\tilde{u}_{\infty}$ is $J$-holomorphic
satisfying

\begin{equation}
\left|\nabla\tilde{u}_{\infty}(0)\right|=1,\label{eq:nonconstant2}
\end{equation}
\begin{equation}
\left\Vert \nabla\tilde{u}_{\infty}\right\Vert \leqq2,
\end{equation}
\begin{equation}
\intop_{B(0,\varepsilon_{k}R_{k})}\hspace{-1.5em}\tilde{u}_{k}^{*}\omega=\intop_{B(z_{k},\varepsilon_{k})}\hspace{-1em}\tilde{u}^{*}\omega\to0\;\mbox{ as \ensuremath{k}\ensuremath{\to}+\ensuremath{\infty}.}
\end{equation}
Thus, $E_{\omega}(\tilde{u}_{\infty})=\intop_{\mathbb{C}}\tilde{u}_{\infty}^{*}\omega=0$.
Moreover, given $R'>0$ and $\phi\in\mathcal{C}$ we have

\[
\intop_{B(0,R')}\hspace{-1em}\tilde{u}_{k}^{*}\left[\phi(r)\sigma\wedge\lambda\right]=\intop_{B(z_{k},R'/R_{k})}\hspace{-1em}\hspace{-1em}\tilde{u}^{*}\left[\phi(r)\sigma\wedge\lambda\right]\rightarrow0,
\]
as $k\to+\infty$. This means $\intop_{B(0,R')}\tilde{u}_{\infty}^{*}\left[\phi(r)\sigma\wedge\lambda\right]=0$,
so $E_{\lambda}(\tilde{u}_{\infty})=0$. Hence, $\tilde{u}_{\infty}$
is constant, contradicting (\ref{eq:nonconstant2}). \end{proof}
\begin{prop}
\label{pro:gradient bound for v holomorphic cylinder}Suppose $J$
is a cylindrical almost complex structure on $\mathbb{R}\times V$.
Let $\tilde{v}:\mathbb{R}^{+}\times S^{1}\to W$ be a $J-$holomorphic
map with respect to the standard complex structure on $\mathbb{R}^{+}\times S^{1}.$
Assume $E(\tilde{v})<+\infty$, then we have 
\[
\left\Vert \nabla\tilde{v}\right\Vert <+\infty,
\]
 where $\left\Vert \nabla\tilde{v}\right\Vert :=\underset{(s,t)\in\mathbb{R}^{+}\times S^{1}}{\sup}\left|\nabla\tilde{v}(s,t)\right|,$
and the norm $|\cdot|$ is computed with respect to the standard product
metric $ds^{2}+dt^{2}$ on $\mathbb{R}^{+}\times S^{1}$ and a translationally
invariant metric $g_{W}$ on $W,$ and $\nabla$ is the Levi-Civita
connection with respect to $g_{W}.$\end{prop}
\begin{proof}
The proof is almost the same as the proof of Proposition \ref{pro:gradient bound for finite energy curve}.\end{proof}
\begin{rem}
Actually, we can see that we can get a gradient bound with respect
to a metric $g_{D}$ on the domain and a translationally invariant
metric $g_{W}$ on $W$, as long as the injectivity radius of $g_{D}$
is bounded away from $0$. 
\end{rem}

\subsection{Subsequence convergence}
\begin{thm}
\label{thm:subsequence convergence to Reeb}Let $J$ be an asymptotically
cylindrical almost complex structure on $\mathbb{R}^{\pm}\times V$,
$\tilde{v}=(a,v):\mathbb{R}^{\pm}\times S^{1}\to\mathbb{R}^{\pm}\times V$
be a $J$-holomorphic curve with $E(\tilde{v})<+\infty.$ Suppose
that $\tilde{v}(\mathbb{R}^{\pm}\times S^{1})$ is unbounded. Then
for any sequence $k_{n}\to+\infty$, there exists a subsequence $k_{n_{i}}$,
such that $v(k_{n_{i}},\cdot)$ converges in $C^{\infty}(S^{1})$
to a map $S^{1}\to V$ given by $t\mapsto x(tT)$, where $x:\mathbb{R}\to V$
is a $|T|$-periodic solution of $\dot{x}=\mathbf{R}_{\infty}(x)$. \end{thm}
\begin{proof}
We prove this theorem for the case $\mathbb{R}^{+}\times V$. The
proof for $\mathbb{R}^{-}\times V$ case can be carried out similarly,
and hence is omitted. By Proposition \ref{pro:gradient bound for v holomorphic cylinder}
we have $\left\Vert \nabla\tilde{v}\right\Vert \leqq C$ for some
$C>0$. Since $\tilde{v}(\mathbb{R}^{+}\times S^{1})$ is not bounded,
there exist a sequence of points $(s_{k},t_{k})\in\mathbb{R}^{+}\times S^{1}$,
such that $\left|a(s_{k},t_{k})\right|\to+\infty$. Now there are
two cases.

\textbf{Case 1}: $a(s_{k},t_{k})\to+\infty$.

Suppose that there exist a sequence of points $(s_{k}^{'},t_{k}^{'})\in\mathbb{R}^{+}\times S^{1}$,
such that $a(s_{k}^{'},t_{k}^{'})<Q$ for some constant $Q$. Pick
a subsequence of $(s_{k},t_{k})$, still called $(s_{k},t_{k}),$
and a subsequence of $(s_{k}^{'},t_{k}^{'}),$ still called $(s_{k}^{'},t_{k}^{'}),$
so that they satisfy $s_{k}^{'}<s_{k}<s_{k+1}^{'}$ for all $k$.
This is possible because $s_{k}\to+\infty$. Since $\left\Vert \nabla\tilde{v}\right\Vert \leqq C$,
we have $a(s_{k}^{'},t)<Q+C$ for $t\in S^{1}$. Consider the compact
manifold $N:=[Q,Q+2C]\times M\subset W=\mathbb{R}^{+}\times V$. Pick
$\phi\in\mathcal{C}$, such that $\phi|_{[Q,Q+2C]}>0$. By Gromov's
Monotonicity (see for example Theorem 1.3 in \cite{Hummel}), there
exists $\iota>0$ such that $\intop_{\tilde{v}([s_{k}^{'},s_{k}]\times S^{1})}\omega+\phi(r)\sigma\wedge\lambda\geqq\iota>0$
for all $k$. This contradicts the fact that $E(\tilde{v})<+\infty$.
Thus $a(s,t)\to+\infty$ uniformly in $t$ as $s\to+\infty$.

Define 
\[
\tilde{v}_{n}(s,t)=(a(s+k_{n},t)-a(k_{n},0),v(s+k_{n},t)).
\]
 Then the sequence $\tilde{v}_{n}(0,0)=(0,v(k_{n},0))$ is bounded.
Since $\tilde{v}$ is $J$-holomorphic, by Lemma \ref{lem:Gromov-Schwarz}
and Ascoli-Arzela Theorem, there exists a subsequence still called
$\tilde{v}_{n}$ converging to $\tilde{v}_{\infty}=(b,v_{\infty}):\mathbb{R}\times S^{1}\to W$
in $C_{loc}^{\infty}$. We know $\tilde{v}_{\infty}$ is $J_{\infty}$-holomorphic.
Define the translation map $\tau_{n}:\mathbb{R}\times S^{1}\to\mathbb{R}\times S^{1}$
by $\tau_{n}(s,t)=(s+k_{n},t)$. Observe

\begin{align}
\intop_{[-R,R]\times S^{1}}\hspace{-1.5em}\tilde{v}_{n}^{*}\omega_{\infty} & =\intop_{[-R+k_{n},R+k_{n}]\times S^{1}}\hspace{-2em}\tilde{v}^{*}\omega+\intop_{[-R+k_{n},R+k_{n}]\times S^{1}}\hspace{-2em}\tilde{v}^{*}(\omega_{\infty}-\omega).\label{eq:w+ w8-w}
\end{align}
For the first term on the right hand side we have 
\begin{equation}
\intop_{[-R+k_{n},R+k_{n}]\times S^{1}}\hspace{-2em}\tilde{v}^{*}\omega\to0\label{eq: integral v*w converge to 0}
\end{equation}
 as $n\to\infty$, since $E_{\omega}(\tilde{v})$ is finite. The second
term satisfies

\begin{align}
\intop_{[-R+k_{n},R+k_{n}]\times S^{1}}\hspace{-2em}\tilde{v}^{*}(\omega_{\infty}-\omega) & \leqq\intop_{[-R+k_{n},R+k_{n}]\times S^{1}}\hspace{-2em}|(\omega_{\infty}-\omega)(\tilde{v}_{s},\tilde{v}_{t})|ds\wedge dt\to0\label{eq:w-w8}
\end{align}
by (AC4) as $n\to+\infty$. 

Combining (\ref{eq:w+ w8-w}), (\ref{eq: integral v*w converge to 0})
and (\ref{eq:w-w8}), we can see $\intop_{[-R,R]\times S^{1}}\tilde{v}_{\infty}^{*}\omega_{\infty}=0$
and hence $E_{\omega_{\infty}}(\tilde{v}_{\infty})=0$, so there exists
a smooth map $f:\mathbb{R}^{2}\to\mathbb{R}$ such that $\tilde{v}_{\infty}=(b,x\circ f)$,
where $x:\mathbb{R}\to V$ is the solution of $\dot{x}=\mathbf{R}_{\infty}(x)$.
Let $\Phi$ be the holomorphic function defined by $\Phi=b+if$. Since
$\left\Vert \nabla\Phi\right\Vert \leqq C$, $\Phi$ is linear. Thus,
$\Phi(s,t)=\alpha(s+it)+\beta$, where $\alpha=T+il,\beta=m+in\in\mathbb{C}$
are constants. But $b(s,t)-b(s,t+1)=0$ implies $l=0$, and $b(0,0)=0$
implies $m=0.$ Thus, 
\begin{equation}
f=Tt+n,\label{eq:f}
\end{equation}
\begin{equation}
b=Ts.
\end{equation}
Therefore, $a_{s}(k_{n},t)\to T$ uniformly in $t$ as $n\to+\infty$
(Recall the notation $\tilde{v}=(a,v)$, $\tilde{v}_{\infty}=(b,v_{\infty})$).
Moreover, we have 

\begin{equation}
\intop_{\{0\}\times S^{1}}\hspace{-1em}\tilde{v}_{\infty}^{*}\lambda_{\infty}=\intop_{\{0\}\times S^{1}}\hspace{-1em}\lambda_{\infty}[(\tilde{v}_{\infty})_{t}]dt=\intop_{\{0\}\times S^{1}}\hspace{-1em}b_{s}dt=T.\label{eq:l}
\end{equation}

\textbf{Claim}: $T\neq0$. 

It follows from the claim and (\ref{eq:f}) that $\tilde{v}_{\infty}$
is not constant. Indeed, by (\ref{eq:f}), $f(s,t+1)=T(t+1)+n$, so
$x(T(t+1)+n)=x(Tt+n).$ Hence, $x$ is $T$-periodic.

\textsl{Proof of Claim}. Suppose $T=0$. Since $a(s,t)\to+\infty$
uniformly in $t$ as $s\to+\infty$, we can pick a subsequence $k_{n_{m}}$
of $k_{n}$ and a sequence $t_{m}\in S^{1}$, such that $a(k_{n_{m+1}},t_{m+1})-a(k_{n_{m}},t_{m})\geqq4C.$
Denote $a(k_{n_{m}},t_{m})$ by $a_{m}$. From $\left\Vert \nabla\tilde{u}\right\Vert \leqq C$
we get 
\begin{equation}
a(k_{n_{m}},t)\in[a_{m}-C,a_{m}+C],\label{eq:a_k}
\end{equation}
\begin{equation}
a(k_{n_{m+1}},t)\geqq a_{m}+3C.\label{eq:a_k+1}
\end{equation}
Let $\psi_{m}:\mathbb{R}\to[0,1]$ be a smooth map, satisfying $\psi_{m}(r)=\frac{1}{7C}(r-a_{m}+\frac{3}{2}C)$
for $r\in[a_{m}-C,a_{m}+5C]$, and $\phi_{m}=\psi_{m}^{'}\in\mathcal{C}$.
We can further require $C>1$, then $\phi_{m}(r)\leqq\frac{1}{7C}<1$.
Observe

\[
\intop_{[k_{n_{m}},k_{n_{m+1}}]\times S^{1}}\hspace{-1.5em}\tilde{v}^{*}d(\psi_{m}(r)\lambda)=\intop_{\{k_{n_{m+1}}\}\times S^{1}}\hspace{-1.5em}\tilde{v}^{*}(\psi_{m}(r)\lambda)-\intop_{\{k_{n_{m}}\}\times S^{1}}\hspace{-1.5em}\tilde{v}^{*}(\psi_{m}(r)\lambda).
\]
While,

\begin{align*}
\left|\intop_{\{k_{n_{m+1}}\}\times S^{1}}\hspace{-1.5em}\tilde{v}^{*}(\psi_{m}(r)\lambda)\right| & =\left|\intop_{\{k_{n_{m+1}}\}\times S^{1}}\hspace{-1.5em}\psi_{m}(\tilde{v})\lambda(\tilde{v}_{t})dt\right|\leqq\intop_{\{k_{n_{m+1}}\}\times S^{1}}\hspace{-1.5em}|\lambda(\tilde{v}_{t})|dt\to T=0,
\end{align*}
as $m\to+\infty.$ Similarly, $\intop_{\{k_{n_{m}}\}\times S^{1}}\tilde{v}^{*}(\psi_{m}(r)\lambda)\to0$.
Thus, by Stokes' theorem 
\begin{equation}
\intop_{[k_{n_{m}},k_{n_{m+1}}]\times S^{1}}\hspace{-2em}\tilde{v}^{*}d(\psi_{m}(r)\lambda)\to0.\label{eq:cusp area}
\end{equation}
Observe 
\begin{align}
 & \intop_{[k_{n_{m}},k_{n_{m+1}}]\times S^{1}}\hspace{-2em}\tilde{v}^{*}(\phi_{m}(r)\sigma\wedge\lambda)\nonumber \\
= & \intop_{[k_{n_{m}},k_{n_{m+1}}]\times S^{1}}\hspace{-2em}\tilde{v}^{*}(\phi_{m}(r)dr\wedge\lambda)+\intop_{[k_{n_{m}},k_{n_{m+1}}]\times S^{1}}\hspace{-2em}\tilde{v}^{*}\left[\phi_{m}(r)\left(\sigma-dr\right)\wedge\lambda\right]\label{eq: sigma wedge lambda}
\end{align}
While, for the first term on the right hand side, we have 
\begin{align}
 & \left|\intop_{[k_{n_{m}},k_{n_{m+1}}]\times S^{1}}\hspace{-2em}\tilde{v}^{*}(\phi_{m}(r)dr\wedge\lambda)\right|\nonumber \\
\leqq & \left|\intop_{[k_{n_{m}},k_{n_{m+1}}]\times S^{1}}\hspace{-2em}\tilde{v}^{*}d(\psi_{m}(r)\lambda)\right|+\intop_{[k_{n_{m}},k_{n_{m+1}}]\times S^{1}}\hspace{-2em}|\tilde{v}^{*}(\psi_{m}(r)d\lambda)|\nonumber \\
\leqq & \left|\intop_{[k_{n_{m}},k_{n_{m+1}}]\times S^{1}}\hspace{-2em}\tilde{v}^{*}d(\psi_{m}(r)\lambda)\right|+\intop_{[k_{n_{m}},k_{n_{m+1}}]\times S^{1}}\hspace{-2em}\tilde{v}^{*}\left(c\omega+c_{m}\sigma\wedge\lambda\right),\label{eq:phi_k}
\end{align}
for some $c>0,$ $c_{m}>0$. The second inequality is due to the fact
that $c\omega+c_{m}\sigma\wedge\lambda$ is positive on all $J$-complex
planes; also since $d\lambda\to d\lambda_{\infty}$ and $i(\frac{\partial}{\partial r})d\lambda_{\infty}=0=i(\mathbf{R}_{\infty})d\lambda_{\infty}$,
we can require that $c$ is independent of $m$ and $c_{m}$ goes
to $0$ as $m\to+\infty$. Similarly, we have

\begin{equation}
\left|\intop_{[k_{n_{m}},k_{n_{m+1}}]\times S^{1}}\hspace{-1.5em}\tilde{v}^{*}\left[\phi_{m}(r)(\sigma-dr)\wedge\lambda\right]\right|\leqq\intop_{[k_{n_{m}},k_{n_{m+1}}]\times S^{1}}\hspace{-1.5em}\tilde{v}^{*}\left[c\omega+c_{m}\sigma\wedge\lambda\right].\label{eq:sigma-dr}
\end{equation}
When $k$ is large, from (\ref{eq: sigma wedge lambda}), (\ref{eq:phi_k})
and (\ref{eq:sigma-dr}) we get

\begin{align}
\nonumber \\
\hspace{-1.5em}\intop_{[k_{n_{m}},k_{n_{m+1}}]\times S^{1}}\hspace{-1.5em}\tilde{v}^{*}(\phi_{m}(r)\sigma\wedge\lambda)\leqq & D\left\{ \left|\intop_{[k_{n_{m}},k_{n_{m+1}}]\times S^{1}}\hspace{-1.5em}\tilde{v}^{*}d(\psi_{m}(r)\lambda)\right|+\intop_{[k_{n_{m}},k_{n_{m+1}}]\times S^{1}}\hspace{-1.5em}\tilde{v}^{*}\omega\right\} ,\label{eq:phi_k<D}
\end{align}
 for some constant $D>0$ which does not depend on $m$ and $\tilde{v}$.
The reason that the term $\intop_{[k_{n_{m}},k_{n_{m+1}}]\times S^{1}}\tilde{v}^{*}\left(c_{m}\sigma\wedge\lambda\right)$
does not show up on the right hand side of (\ref{eq:phi_k<D}) is
because that it is absorbed by the left hand side since $\phi_{m}|_{[k_{n_{m}},k_{n_{m+1}}]\times S^{1}}=1/7$.
Since $E_{\omega}(\tilde{v})$ is finite, $\intop_{[k_{n_{m}},k_{n_{m+1}}]\times S^{1}}\tilde{v}^{*}\omega$
goes to $0$. Together with (\ref{eq:cusp area}), we get 
\[
\intop_{[k_{n_{m}},k_{n_{m+1}}]\times S^{1}}\tilde{v}^{*}(\phi_{m}(r)dr\wedge\lambda)\to0.
\]
Summing up, 
\begin{equation}
\intop_{[k_{n_{m}},k_{n_{m+1}}]\times S^{1}}\hspace{-1.5em}\tilde{v}^{*}(\omega+\phi_{m}(r)dr\wedge\lambda)\to0\label{eq:cusp area 2}
\end{equation}
 as $m\to+\infty$.

Now consider $N_{m}=[a_{m}+C,a_{m}+3C]\times V\subseteq W$ with an
almost complex structure $J_{m}:=J|_{N_{m}}$ and a non-degenerate
2-form $\Omega_{m}:=\omega+\phi_{m}(r)\sigma\wedge\lambda|_{N_{m}}$.
Because of the asymptotic condition, we can find uniform constants
$C_{0}$, $r_{0}>0$ such that by the Gromov's Monotonicity any $J_{m}$-holomorphic
curve $h_{m}:(S,j)\to(N_{m},J_{m})$ where $(S,j)$ is a Riemann surface
with boundary, and if the boundary $h_{m}(\partial S)$ is contained
in the complement of the ball $B(h_{m}(s_{0}),r)$ where $s_{0}\in\mbox{Int}S_{m}$
and $r<r_{0}$, then we have $\intop_{h_{m}(S)\cap B(h_{m}(s_{0}),r))}\Omega_{m}\geqq C_{0}r^{2}$.
By (\ref{eq:a_k}) and (\ref{eq:a_k+1}) we can see $\tilde{u}(k_{n_{m}},S^{1})\cap\mbox{Int}N_{m}=\emptyset$
and $\tilde{u}(k_{n_{m+1}},S^{1})\cap\mbox{Int}N_{m}=\emptyset$.
This contradicts  (\ref{eq:cusp area 2}). Thus, $T\neq0$.

\textbf{Case 2.} $a(s_{k},t_{k})\to-\infty$.

We deal with it similarly. \end{proof}
\begin{cor}
\label{coro:a(s,t)-Ts  and zout converges to 0.}Under the assumption
of Theorem \ref{thm:subsequence convergence to Reeb}, there exists
a number $T>0$ such that as $s\to\pm\infty,$
\begin{equation}
\partial^{\beta}[a(s,t)-Ts]\to0\label{eq:cat}
\end{equation}
 uniformly in t, provided $\beta=(\beta_{1},\beta_{2})\in\mathbb{Z}_{\geqq0}\times\mathbb{Z}_{\geqq0}$
and $|\beta|=\beta_{1}+\beta_{2}\geqq1$.\end{cor}
\begin{proof}
By Theorem \ref{thm:subsequence convergence to Reeb}, there exist
a number $T>0$ and a sequence of numbers $s_{k}^{'}$ such that $s_{k}^{'}\to+\infty$
and $v(s_{k}^{'},\cdot)\to x(T\cdot),$ for some $T$-periodic orbit
$x$ of $\mathbf{R}_{\infty}.$ Suppose \eqref{eq:cat} is not true
for this $T$, then there exists a sequence of points $(s_{k},t_{k})$
such that $s_{k}\to+\infty$ and $\partial^{\beta}[a(s,t)-Ts]|_{(s_{k},t_{k})}\to c$
as $k\to+\infty$ for some $|\beta|\geqq1$, where $c$ is a non-zero
constant ($\pm\infty$ included). Define $\bar{a}_{k}(s,t):=a(s+s_{k},t+t_{k})-a(s_{k},t_{k})$,
and then $\bar{a}_{k}(0,0)=0$. From the proof of Theorem \ref{thm:subsequence convergence to Reeb}
we get a subsequence of $k$, still called $k$, and a $T'$-periodic
orbit $x'$ of $\mathbf{R}_{\infty},$ such that $\bar{a}_{k}\to T's$
in $C_{loc}^{\infty}(\mathbb{R}^{+}\times S^{1},\mathbb{R}).$ By
a straight forward modification of the proof of Proposition 2.1 in
\cite{Finite energy plane} to the Morse-Bott case, we can show that
$x'$ and $x$ lie in the same component of $N_{T}$ (see Definition
\ref{def:Morse Bott}) and in particular $T'=T$. Thus, 
\begin{align*}
\partial^{\beta}[a(s,t)-Ts]|_{(s_{k},t_{k})} & =\partial^{\beta}[a(s+s_{k},t+t_{k})-a(s_{k},t_{k})-Ts]|_{(0,0)}\\
 & =\partial^{\beta}(\bar{a}_{k}(s,t)-Ts)|_{(0,0)}\\
 & \to0,
\end{align*}
 which contradicts the assumption.
\end{proof}
To prove Theorem \ref{thm:converge to reeb orbit} and Theorem \ref{thm: exponential convergence},
we need to get the exponential decay estimates.

\subsection{\label{sub:Exponential-decay}Exponential decay estimates}

In this subsection, we will follow the schemes in \cite{morse bott}
to prove Theorem \ref{thm:converge to reeb orbit} and Theorem \ref{thm: exponential convergence}.
The strategy is as follows: Firstly, we pick a neighborhood $U$ of
the orbit $\gamma$, restrict the $J$-holomorphic curve $\tilde{u}$
to a sequence of cylinders inside the domain so that the images lie
in the neighborhood and satisfy certain inequalities, and estimate
the behaviors of each finite cylinder by the behaviors of boundaries
of the cylinder. Secondly, since we have a sequence of circles in
the domain whose images lie in $U,$ we get the cylinders bounded
by the circles also lie in $U$ based on the estimates. We also show
that near the end of the domain the $\tilde{u}$ satisfies the inequalities.
Once these are achieved, Theorem \ref{thm:converge to reeb orbit}
and Theorem \ref{thm: exponential convergence} follow easily.

In order to study the $J$-holomorphic curve equation around $\gamma$,
we need introduce a good coordinate chart around a neighborhood of
$\gamma.$
\begin{lem}
\label{lem:choosing nice coordinate}\cite{compactness} Suppose that
$J_{\infty}$ is a cylindrical almost complex structure of the Morse-Bott
type on $\mathbb{R}^{+}\times V$ at $\infty$. Let $N$ be a component
of the set $N_{T}\subset V$ (see Definition \ref{def:Morse Bott}),
and $\gamma$ be one of the orbits from $N$.

a) if $T$ is the minimal period of $\gamma$ then there exists a
neighborhood $U\supset\gamma$ in $V$ such that $U\cap N$ is invariant
under the flow of $\mathbf{R}_{\infty}$ and one finds coordinates
$(\vartheta,x_{1},...,x_{n},y_{1},...,y_{n})$ of $U$ such that 
\[
N=\{x_{1},...,x_{p}=0,y_{1},...,y_{q}=0\},
\]
for $0\leqq p,q\leqq n,$ 
\[
\mathbf{R}_{\infty}|_{N}=\frac{\partial}{\partial\vartheta},
\]
and
\[
\omega_{\infty}|_{N}=\omega_{0}|_{N},
\]
 where $\omega_{0}=\sum_{i=1}^{n}dx_{i}\wedge dy_{i}$.

b) if $\gamma$ is a $m$-multiple of a trajectory $\bar{\gamma}$
of a minimal period $\frac{T}{m}$ there exists a tubular neighborhood
$\bar{U}$ of $\bar{\gamma}$ such that its $m$-multiple cover $U$
together with all the structures induced by the covering map from
$U\to\bar{U}$ from the corresponding objects on $\bar{U}$ satisfy
the properties of the part a).\end{lem}
\begin{proof}
Refer to Lemma A.1 in \cite{compactness}.
\end{proof}

Using this coordinate chart, we can work locally in $U\subset(\mathbb{R}/T\mathbb{Z})\times\mathbb{R}^{2n}$,
and make $T$ the minimal period of $\gamma$. Denote by $z_{in}$
the coordinate $(x_{1},...,x_{p},y_{1},...,y_{q})$ and by $z_{out}$
the coordinate $(x_{n-p+1},...,x_{n},y_{n-p+1},...,y_{n})$. We can
easily see following lemma about behavior of a $J$-holomorphic curve
in the $z_{out}$ direction.

\begin{lem}
Let $J$ be an asymptotically cylindrical almost complex structure
on $W=\mathbb{R}^{+}\times V$, and $\tilde{u}$ be a finite Hofer
energy $J$-holomorphic curve from $\mathbb{R}^{+}\times S^{1}$ to
$W$. Suppose $[m_{k},n_{k}]$ is a sequence of intervals in $\mathbb{R}^{+}$
with $m_{k}\to+\infty$ and $\tilde{u}([m_{k},n_{k}]\times S^{1})\subset U$,
then we have as $k\to+\infty,$ 
\[
\underset{(s,t)\in[m_{k},n_{k}]\times S^{1}}{\sup}\left|\partial^{\beta}z_{out}(s,t)\right|\to0
\]
 for all $\beta\in\mathbb{Z}_{\geqq0}\times\mathbb{Z}_{\geqq0}$.
\end{lem}

\begin{proof}
The proof is very similar to the proof of Corollary \ref{coro:a(s,t)-Ts  and zout converges to 0.},
so we omit it here.
\end{proof}

Let's study the $J$-holomorphic curve equation in $\mathbb{R}^{+}\times U\subset\mathbb{R}^{+}\times(\mathbb{R}/T\mathbb{Z})\times\mathbb{R}^{2n}$.
Denote $\theta:=[s_{0},s_{1}]\times S^{1}$ for some $s_{0}<s_{1}$
and let $\tilde{u}=(a,\vartheta,z):\theta\to\mathbb{R}\times U$ be
a $J$-holomorphic curve, then we have 

\begin{equation}
(a_{s},\vartheta_{s},z_{s})+J(\tilde{u})(a_{t},\vartheta_{t},z_{t})=0.\label{eq:Jhol local}
\end{equation}

Rewrite this equation according to its $z$-component, $\vartheta$-component,
and $a$-component we get%
\footnote{From (\ref{eq:zs+Mzt+Szout}) we can see that if we require $z$,
$z_{s}$ and $z_{t}$ decay exponentially, $L$ has to decay exponentially.
The condition $f_{s}^{*}J\to J_{\infty}$ in $C_{loc}^{\infty}$ is
not enough to guarantee that $L$ decays exponentially fast. %
} 

\begin{equation}
z_{s}+Mz_{t}+Sz_{out}+L=0,\label{eq:zs+Mzt+Szout}
\end{equation}
\begin{equation}
a_{s}-\vartheta_{t}+Bz_{out}+B'z_{t}+N=0,\label{eq:a' with respect to s}
\end{equation}
\begin{equation}
a_{t}+\vartheta_{s}+Cz_{out}+C'z_{s}+O=0,\label{eq:a' with respect to t}
\end{equation}
 where $M,S,B,B',C,C'$ depend on $a(s,t),\vartheta(s,t),z(s,t),$
and are bounded by a constant $C_{0},$ and $L,N,O$ depend on $a(s,t),\vartheta(s,t),z(s,t)$
and are bounded by $C_{0}e^{-\delta a}.$

Define an operator $A(s):W^{1,2}(S^{1},\mathbb{R}^{2n})\to L^{2}(S^{1},\mathbb{R}^{2n})$
by 
\[
(A(s)w)(t)=-M(\tilde{u}(s,t))w_{t}(t)-S(\tilde{u}(s,t))w_{out}(t),
\]
 then by (\ref{eq:zs+Mzt+Szout}) we get 
\begin{equation}
A(s)z(s,\cdot)=z_{s}+L.\label{eq: A(s)z...z_s+L}
\end{equation}

Notice that $A(s)$ depends on the map $\tilde{u}=(a,\vartheta,z_{in},z_{out}).$
If we do not use the original $J$-holomorphic curve $\tilde{u}$,
instead, we substitute $\vartheta(s,t)=\vartheta(s_{0},0)+Tt,$ $a(s,t)=Ts,$
$z_{out}(s,t)=0,$ and $z_{in}(s,t)=z_{in}(s_{0},t),$ then we get
another operator denoted by $\tilde{A}(s)$. We can easily see that
$\underset{s\to+\infty}{\lim}\tilde{A}(s)$ exists and denote the
limiting operator by $A_{0}$. Similarly we get two matrices $M_{0}(t)$
and $S_{0}(t)$, and we have 
\[
M_{0}(t)^{2}=-id,
\]
 and 
\begin{equation}
(A_{0}w)(t)=-M_{0}(t)w_{t}(t)-S_{0}(t)w_{out}.\label{eq:def of A0}
\end{equation}

Consider an inner product on $L^{2}(S^{1},\mathbb{R}^{2n})$ given
by 
\begin{equation}
\langle u,v\rangle_{0}=\int_{0}^{1}\langle u,-J_{0}M_{0}v\rangle dt,\label{eq:inner product 1}
\end{equation}
 where the inner product is given by $\langle\cdot,\cdot\rangle=\omega_{0}(\cdot,J_{0}\cdot),$
and $J_{0}$ is the standard complex structure on $\mathbb{R}^{2n}.$
With respect to the inner product $\langle\cdot,\cdot\rangle_{0}$,
one can check directly that $M_{0}$ is anti-symmetric, and $A_{0}$
is self-adjoint.

\begin{rem}
$A_{0}$ is injective iff $\gamma$ is non-degenerate. 
\end{rem}
It is not hard to see that $\ker A_{0}$ consists of the constant
vector fields in $N$ along $\gamma_{0}$. Let's denote by $P_{0}$
the projection onto $\ker A_{0}$ with respect to $\langle\cdot,\cdot\rangle_{0}$,
and let $Q_{0}:=I-P_{0}$. It is easy to check that $Q_{0}$ satisfies:
\begin{lem}
$\left(Q_{0}w\right)_{t}=w_{t}$, $\left(Q_{0}w\right)_{s}=Q_{0}w_{s}$,
$\left(Q_{0}w\right)_{out}=w_{out}$ and $Q_{0}A_{0}=A_{0}Q_{0}.$
\end{lem}
The following lemma will be needed in proving Lemma \ref{lem:g''>cg-ce^=00007B-cs=00007D}.

\begin{lem}
\label{lem:|AQxi|>C|Qxi|+C|Qxi|t}There exists a constant $C>0$ such
that 
\[
\|A_{0}Q_{0}w\|_{0}\geqq C\left(\|Q_{0}w\|_{0}+\|\left(Q_{0}w\right)_{t}\|_{0}\right)
\]
 for $w\in W^{1,2}(S^{1},\mathbb{R}^{2n}),$ where the norm $\Vert\cdot\Vert_{0}$
is defined using the inner product $\left\langle \cdot,\cdot\right\rangle _{0}$. \end{lem}
\begin{proof}
To prove the lemma we only need to prove $\|A_{0}Q_{0}w\|_{0}\geqq C'\|Q_{0}w\|_{0}$
for some $C'>0,$ because by definition we have

\begin{equation}
A_{0}Q_{0}w=-M_{0}\left(Q_{0}w\right)_{t}-S_{0}Q_{0}w.
\end{equation}
 Suppose to the contrary, there exists $\varepsilon_{n}\to0$ and
$w_{n}\in W^{1,2}(S^{1},\mathbb{R}^{2n})$ satisfying $\|Q_{0}w_{n}\|_{0}=1$
and $\|A_{0}Q_{0}w_{n}\|_{0}\leqq\varepsilon_{n}$. Then we have
\[
\left\Vert \left(Q_{0}w_{n}\right)_{t}\right\Vert _{0}\leqq\left\Vert M_{0}A_{0}Q_{0}w_{n}\right\Vert _{0}+\left\Vert M_{0}S_{0}Q_{0}w_{n}\right\Vert _{0}\leqq\varepsilon_{n}+C''.
\]
 Therefore, $Q_{0}w_{n}$ is bounded in $W^{1,2}(S^{1},\mathbb{R}^{2n}).$
Since $W^{1,2}(S^{1},\mathbb{R}^{2n})$ embeds compactly in $L^{2}(S^{1},\mathbb{R}^{2n})$
we get a subsequence of $w_{n}$, still denoted by $w_{n}$, such
that $Q_{0}w_{n}$ is a Cauchy sequence in $L^{2}(S^{1},\mathbb{R}^{2n})$.
But it is easy to see that $\left(Q_{0}w_{n}\right)_{t}$ is also
a Cauchy sequence in $L^{2}(S^{1},\mathbb{R}^{2n}).$ Therefore, $Q_{0}w_{n}$
converges to some $\eta$ in $W^{1,2}(S^{1},\mathbb{R}^{2n}),$ so
$\eta\in\ker A_{0}.$ Because $\eta$ also lies in the orthogonal
complement of $\ker A_{0}$, $\eta$ has to be $0,$ which contradicts
the fact $\|\eta\|_{0}=\underset{n\to0}{\lim}\|Q_{0}w_{n}\|_{0}=1.$ 
\end{proof}
Denote $g_{0}(s):=\frac{1}{2}\|Q_{0}z(s)\|_{0}^{2}$ and $\kappa_{0}(s):=(\vartheta(s_{0},0)-\vartheta(s,0),z_{in}(s_{0},0)-z_{in}(s,0))$,
and then we have
\begin{lem}
\label{lem:g''>cg-ce^=00007B-cs=00007D}There exist $\delta=\delta(\beta)>0$,
$\flat=\flat(\beta)>0$ and $\bar{\kappa}=\bar{\kappa}(\beta)>0$
such that, if for any multi-indices $\beta$ 
\[
a(s_{0},0)\geqq\flat,
\]
\[
|\kappa_{0}(s_{0})|\leqq\bar{\kappa},
\]
\[
\underset{(s,t)\in\theta}{\sup}\left|\partial^{\beta}z_{out}(s,t)\right|\leqq\delta,
\]
 and for any multi-indices $\beta$ with \textup{$|\beta|>0,$} 
\[
\underset{(s,t)\in\theta}{\sup}\left|\partial^{\beta}(a(s,t)-Ts)\right|\leqq\delta,
\]
\[
\underset{(s,t)\in\theta}{\sup}\left|\partial^{\beta}(\vartheta(s,t)-Tt)\right|\leqq\delta,
\]
\[
\underset{(s,t)\in\theta}{\sup}\left|\partial^{\beta}z_{in}(s,t)\right|\leqq\delta,
\]
 \textup{then for $s\in[s_{0},\mathfrak{s}]$, we have}

\[
g_{0}^{''}(s)\geqq c^{2}g_{0}(s)-c_{2}e^{-c_{1}(s-s_{0})},
\]
 where 
\[
\mathfrak{s}:=\sup\left\{ s\in[s_{0},s_{1}]:|\kappa_{0}(s')|\leqq\bar{\kappa}\mbox{ for all }s'\in[s_{0},s]\right\} ,
\]
and $c,c_{1},c_{2}>0$ are constants independent of $s_{0}$ and $s_{1}$. \end{lem}
\begin{proof}
All constants in the proof may depend on $\beta.$ Notice that from
the assumption we have 
\[
\underset{(s,t)\in\theta}{\sup}\left|\partial^{\beta}(\vartheta(s,t)-\vartheta(s,0)-Tt)\right|\leqq\delta,
\]
\[
\underset{(s,t)\in\theta}{\sup}\left|\partial^{\beta}(z_{in}(s,t)-z_{in}(s,0))\right|\leqq\delta,
\]
 for all multi-indices $\beta.$ 

Define an operator $\bar{A}(s)w=-\bar{M}(\tilde{u}(s,t))w_{t}(t)-\bar{S}(\tilde{u}(s,t))w_{out}(t)$
in the same way as $A(s)$ but using $J_{\infty}$ instead of $J.$

From (\ref{eq: A(s)z...z_s+L}) we get 
\begin{equation}
z_{s}=A_{0}z+(\Delta_{0}+\tilde{\Delta}_{0}\kappa_{0})z_{t}+(\hat{\Delta}_{0}+\bar{\Delta}_{0}\kappa_{0})z_{out}+[A(s)-\bar{A}(s)]z-L.\label{eq:zs=00003DA0z+Delta-L}
\end{equation}
 Applying $Q_{0}$ to (\ref{eq:zs=00003DA0z+Delta-L}) gives us 
\begin{align}
(Q_{0}z)_{s} & =A_{0}Q_{0}z+Q_{0}(\Delta_{0}+\tilde{\Delta}_{0}\kappa_{0})(Q_{0}z)_{t}+Q_{0}(\hat{\Delta}_{0}+\bar{\Delta}_{0}\kappa_{0})(Q_{0}z)_{out}\nonumber \\
 & \quad+Q_{0}[A(s)-\bar{A}(s)]z-Q_{0}L,\label{eq:Qzs=00003DA0Qz+QDelta}
\end{align}
 where $\Delta_{0}=\bar{M}_{0}-\bar{M}$ and $\hat{\Delta}_{0}=\bar{S}_{0}-\bar{S}$
satisfying for any multi-indices $\beta$
\[
\underset{(s,t)\in\theta}{\sup}\left|\partial^{\beta}\Delta_{0}(s,t)\right|\leqq C\delta,
\]
\[
\underset{(s,t)\in\theta}{\sup}\left|\partial^{\beta}\hat{\Delta}_{0}(s,t)\right|\leqq C\delta,
\]
 and $\tilde{\Delta}_{0}\kappa_{0}=M_{0}-\bar{M}_{0}$ and $\bar{\Delta}_{0}\kappa_{0}=S_{0}-\bar{S}_{0}$
satisfying for any multi-indices $\beta$ 
\[
\underset{(s,t)\in\theta}{\sup}\left|\partial^{\beta}\tilde{\Delta}_{0}(s,t)\right|\leqq C,
\]

\[
\underset{(s,t)\in\theta}{\sup}\left|\partial^{\beta}\bar{\Delta}_{0}(s,t)\right|\leqq C.
\]
 We can require $0<\delta<\frac{T}{2}$, and we get
\[
a(s,t)\geqq a(s_{0},0)+\frac{T}{2}(s-s_{0})-\delta\geqq(\flat-\delta)+\frac{T}{2}(s-s_{0}).
\]
 Because $J$ is an asymptotically cylindrical almost complex structure,
we get

\[
\Vert Q_{0}L\Vert_{0}\leqq c_{0}e^{-c_{0}^{'}(\flat-\delta)}e^{-c_{0}^{'}\frac{T}{2}(s-s_{0})}
\]
 for some constants $c_{0},c_{0}^{'}>0$. Denote $c_{1}:=c_{0}^{'}\frac{T}{2}$
and $c_{2}:=c_{0}e^{-c_{0}^{'}(\flat-\delta)}$, and then we have

\[
\Vert Q_{0}L\Vert_{0}\leqq c_{2}e^{-c_{1}(s-s_{0})}.
\]
 We also have 
\begin{equation}
\left\Vert \left\{ \partial^{\beta}[A(s)-\bar{A}(s)]\right\} z\right\Vert _{0}\leqq c_{2}e^{-c_{1}(s-s_{0})}\left\Vert Q_{0}z\right\Vert _{0,W^{1,2}}\label{eq:bar=00007BA=00007D-A}
\end{equation}
 for multi-indices $\beta$, by picking $c_{0}$ larger if necessary.

Now we are ready to estimate $g_{0}^{''}(s)$. Obviously we have 
\[
g_{0}^{''}(s)\geqq\langle Q_{0}z_{ss},Q_{0}z\rangle_{0}.
\]
Now let's compute the right hand side of the above inequality. Differentiate
(\ref{eq:Qzs=00003DA0Qz+QDelta}) with respect to $s$, we obtain
\begin{align*}
(Q_{0}z)_{ss} & =A_{0}Q_{0}z_{s}+Q_{0}(\Delta_{0}+\tilde{\Delta}_{0}\kappa_{0})(Q_{0}z)_{st}+Q_{0}(\Delta_{0}+\tilde{\Delta}_{0}\kappa_{0})_{s}(Q_{0}z)_{t}\\
 & \quad+Q_{0}(\hat{\Delta}_{0}+\bar{\Delta}_{0}\kappa_{0})(Q_{0}z_{s})_{out}+Q_{0}(\hat{\Delta}_{0}+\bar{\Delta}_{0}\kappa_{0})_{s}(Q_{0}z)_{out}\\
 & \quad+Q_{0}[A(s)-\bar{A}(s)]_{s}z+Q_{0}[A(s)-\bar{A}(s)]z_{s}-Q_{0}L_{s},
\end{align*}
Thus we get $\langle Q_{0}z_{ss},Q_{0}z\rangle_{0}$ contains 8 terms.
When we are estimating these terms, each time we see $Q_{0}z_{s}$,
we replace it using (\ref{eq:Qzs=00003DA0Qz+QDelta}). A straightforward
calculation using Lemma \ref{lem:|AQxi|>C|Qxi|+C|Qxi|t} and the fact
that 
\[
-c_{2}e^{-c_{1}(s-s_{0})}\Vert Q_{0}z\Vert_{0,W^{1,2}}\geqq-c_{2}e^{-c_{1}(s-s_{0})}-c_{2}e^{-c_{1}(s-s_{0})}\Vert Q_{0}z\Vert_{0,W^{1,2}}^{2}
\]
gives us 
\[
g_{0}^{''}(s)\geqq(1-10C\delta-10C|\kappa_{0}|-10Cc_{2}e^{-c_{1}(s-s_{0})})g_{0}(s)-c_{2}e^{-c_{1}(s-s_{0})}.
\]
 From the definition of $c_{2}$ we can see that if $\flat$ is large
enough, $c_{2}$ can be very close to $0$. Therefore,
\[
g_{0}^{''}(s)\geqq c^{2}g_{0}(s)-c_{2}e^{-c_{1}(s-s_{0})}.
\]
 We can require further that $c_{1}>c>0.$
\end{proof}
Based on Lemma \ref{lem:g''>cg-ce^=00007B-cs=00007D}, we can easily
get
\begin{lem}
\label{lem:g(s)<max=00007Bg(s0),g(s*)=00007D} Under the same assumption
as in Lemma \ref{lem:g''>cg-ce^=00007B-cs=00007D}, we have 
\begin{align*}
g_{0}(s) & \leqq\max\{g_{0}(s_{0}),g_{0}(\mathfrak{s})\}\frac{\cosh\left[c\left(s-\frac{s_{0}+\mathfrak{s}}{2}\right)\right]}{\cosh\left(c\frac{\mathfrak{s}-s_{0}}{2}\right)}+\frac{c_{2}}{c_{1}^{2}-c^{2}}\frac{\sinh(c(\mathfrak{s}-s))}{\sinh(c(\mathfrak{s}-s_{0}))},
\end{align*}
 for $s_{0}\leqq s\leqq\mathfrak{s}$.\end{lem}
\begin{proof}
Let 
\begin{align*}
h(s) & :=\max\{g_{0}(s_{0}),g_{0}(\mathfrak{s})\}\frac{\cosh\left[c\left(s-\frac{s_{0}+\mathfrak{s}}{2}\right)\right]}{\cosh\left(c\frac{\mathfrak{s}-s_{0}}{2}\right)}\\
 & +\frac{c_{2}}{c_{1}^{2}-c^{2}}\frac{1}{\sinh(c(\mathfrak{s}-s_{0}))}\left\{ \sinh(c(\mathfrak{s}-s))\vphantom{+e^{-c_{1}(\mathfrak{sxxx}-s_{0})}}\right.\\
 & \left.+e^{-c_{1}(\mathfrak{s}-s_{0})}\sinh(c(s-s_{0}))-e^{-c_{1}(s-s_{0})}\sinh(c(\mathfrak{s}-s_{0}))\right\} ,
\end{align*}
 then $h(s)$ satisfies:

\begin{equation}
\begin{cases}
h''(s)-c^{2}h(s)=-c_{2}e^{-c_{1}(s-s_{0})}\\
h(s_{0})=\max\{g_{0}(s_{0}),g_{0}(\mathfrak{s})\}\\
h(\mathfrak{s})=\max\{g_{0}(s_{0}),g_{0}(\mathfrak{s})\}
\end{cases}
\end{equation}
 Let $l(s):=g_{0}(s)-h(s)$, then $l(s)$ satisfies

\begin{equation}
\begin{cases}
l''(s)-c^{2}l(s)\geqq0 & \quad\quad\qquad\qquad\\
l(s_{0})\leqq0\\
l(\mathfrak{s})\leqq0
\end{cases}\label{eq:l(s)}
\end{equation}
 Then by Maximal principle we get $l(s)\leqq0$ for $s_{0}\leqq s\leqq\mathfrak{s}$.
Then the lemma follows from the fact that 
\[
e^{-c_{1}(\mathfrak{s}-s_{0})}\sinh(c(s-s_{0}))-e^{-c_{1}(s-s_{0})}\sinh(c(\mathfrak{s}-s_{0}))\leqq0.
\]

\end{proof}
Now let's study the component $z_{in}.$
\begin{lem}
\label{lem: <z(s),e>-<z(s0),e>}Let $e$ be a unit vector in $\mathbb{R}^{2n}$
with $e_{out}=0$. Under the assumption of Lemma \ref{lem:g''>cg-ce^=00007B-cs=00007D}
and for $s\in[s_{0},\mathfrak{s}]$, we have 
\[
\left|\left\langle z(s),e\right\rangle _{0}-\left\langle z(s_{0}),e\right\rangle _{0}\right|\leqq\frac{8C}{c}\max(\Vert Q_{0}z(s_{0})\Vert_{0},\Vert Q_{0}z(\mathfrak{s})\Vert_{0})+o(c_{2}),
\]
 where $o(c_{2})$ satisfies $\underset{c_{2}\to0}{\lim}o(c_{2})=0,$
and $C$ is a constant independent of $s_{0}$ and $s_{1}$.\end{lem}
\begin{proof}
The inner product of the Cauchy-Riemann equation (\ref{eq:zs+Mzt+Szout})
with $e$ gives

\[
\frac{d}{ds}\left\langle z,e\right\rangle _{0}+\left\langle Mz_{t},e\right\rangle _{0}+\left\langle Sz_{out},e\right\rangle _{0}+\left\langle L,e\right\rangle _{0}=0.
\]
 From 
\begin{align*}
\left\langle Mz_{t},e\right\rangle _{0} & =\int_{0}^{1}\omega_{0}(M\left(Q_{0}z\right)_{t},M_{0}e)dt\\
 & =-\int_{0}^{1}\omega_{0}(M_{t}Q_{0}z,M_{0}e)dt-\int_{0}^{1}\omega_{0}\left(MQ_{0}z,\left(M_{0}\right)_{t}e\right)dt
\end{align*}
 we can see 
\[
\left|\left\langle Mz_{t},e\right\rangle _{0}\right|\leqq C\Vert Q_{0}z\Vert_{0}.
\]
 Together with the facts $\left|\left\langle Sz_{out},e\right\rangle _{0}\right|\leqq C\Vert Q_{0}z\Vert_{0}$
and $\left|\left\langle L,e\right\rangle _{0}\right|\leqq c_{2}e^{-c_{1}(s-s_{0})}$
we get 
\begin{align*}
\left\langle z(s),e\right\rangle _{0}-\left\langle z(s_{0}),e\right\rangle _{0} & \leqq\int_{s_{0}}^{s}\left[2C\Vert Q_{0}z(\mathfrak{x})\Vert_{0}+c_{2}e^{-c_{1}(\mathfrak{x}-s_{0})}\right]d\mathfrak{x}\\
 & \leqq2C\int_{s_{0}}^{s}\sqrt{2g_{0}(\mathfrak{x})}d\mathfrak{x}+\frac{c_{2}}{c_{1}}.
\end{align*}

Now a straightforward calculation using Lemma \ref{lem:g(s)<max=00007Bg(s0),g(s*)=00007D}
and the fact $\sqrt{\cosh u}<\sqrt{2}\cosh\left(\frac{u}{2}\right)$
finishes the proof.\end{proof}
\begin{rem}
By requiring that $\flat$ is sufficiently large, we can make $c_{2}$
sufficiently small. 
\end{rem}
Now let's estimate the derivatives of $z$.
\begin{lem}
\label{lem:higher derivative estimate}There exist $\delta=\delta(\beta)>0$,
$\flat=\flat(\beta)>0$ and $\bar{\kappa}=\bar{\kappa}(\beta)>0$
such that, if for multi-indices $\beta$

\[
\underset{(s,t)\in\theta}{\sup}\left|\partial^{\beta}z_{out}(s,t)\right|\leqq\delta,
\]
\[
a(s_{0},0)\geqq\flat,
\]
 and for those multi-indices $\beta$ with \textup{$|\beta|>0,$}
\[
\underset{(s,t)\in\theta}{\sup}\left|\partial^{\beta}(a(s,t)-Ts)\right|\leqq\delta,
\]
\[
\underset{(s,t)\in\theta}{\sup}\left|\partial^{\beta}(\vartheta(s,t)-t)\right|\leqq\delta,
\]
\[
\underset{(s,t)\in\theta}{\sup}\left|\partial^{\beta}z_{in}(s,t)\right|\leqq\delta,
\]
\textup{ then for $s\in[s_{0},\mathfrak{s}],$ we have}
\begin{align*}
\Vert\partial^{\beta}z(s)\Vert_{0} & \leqq C_{\beta}\underset{|\beta'|\leqq|\beta|}{\max}\left\{ \Vert Q_{0}\partial^{\beta'}z(s_{0})\Vert_{0},\Vert Q_{0}\partial^{\beta'}z(\mathfrak{s})\Vert_{0}\right\} \sqrt{\frac{\cosh\left(c_{1}\left(s-\frac{s_{0}+\mathfrak{s}}{2}\right)\right)}{\cosh\left(c_{1}\left(\frac{s_{0}-\mathfrak{s}}{2}\right)\right)}}\\
 & \quad+D_{\beta}(c_{2})\sqrt{\frac{\sinh(c(\mathfrak{s}-s))}{\sinh(c(\mathfrak{s}-s_{0}))}}+c_{2}e^{-c_{1}(s-s_{0})},
\end{align*}
 where
\[
\mathfrak{s}:=\sup\left\{ s\in[s_{0},s_{1}]:|\kappa_{0}(s')|\leqq\bar{\kappa}\mbox{ for all }s'\in[s_{0},s]\right\} ,
\]
 and $C_{\beta},c_{1}>0$ are constants independent of $s_{0}$ and
$s_{1}$, and $D_{\beta}(c_{2})$ is a function of $c_{2}$ independent
of $s_{0}$ and $s_{1}$, satisfying $\underset{c_{2}\to0}{\lim}C^{\beta}(c_{2})=0,$
and \textup{$l$ is the integer in Definition \ref{def: asympt cylindrical}.}\end{lem}
\begin{proof}
Let's prove the estimate for $|\beta|=1$. The proof of estimates
of higher derivatives is almost the same. Refer to Lemma A.6 in \cite{compactness}
for the estimates for all derivatives in the Cylindrical case.

Equation (\ref{eq:zs=00003DA0z+Delta-L}) can be rewritten as
\begin{equation}
z_{s}=A_{0}z+\dot{\Delta}z_{t}+\ddot{\Delta}z_{out}+\dddot{\Delta}z-L,\label{eq:zs=00003DA0z+...for higher derivative estimates}
\end{equation}
 with $\dot{\Delta}=\Delta_{0}+\tilde{\Delta}_{0}\kappa_{0},$ $\ddot{\Delta}=\hat{\Delta}_{0}+\bar{\Delta}_{0}\kappa_{0},$
and $\dddot{\Delta}=[A(s)-\bar{A}(s)].$ Denote $\mathcal{W}:=\left(Q_{0}z,\frac{\partial}{\partial s}\left(Q_{0}z\right),A_{0}Q_{0}z,\frac{\partial}{\partial s}\left(A_{0}Q_{0}z\right)\right),$
then $\mathcal{W}$ satisfies 

\[
\mathcal{W}_{s}=\mathcal{A}_{0}\mathcal{W}+\mathcal{Q}_{0}\mathbf{\dot{\Delta}}\mathcal{W}_{t}+\mathcal{Q}_{0}\mathbf{\ddot{\Delta}}\mathcal{W}_{out}+\mathbf{\dddot{\Delta}}\mathcal{W}-\mathbf{\mathcal{L}},
\]
 where $\mathcal{A}_{0}=diag(A_{0},A_{0},A_{0},A_{0})$, $\mathcal{Q}_{0}=diag(Q_{0},Q_{0},Q_{0},Q_{0})$,
and $\mathbf{\dot{\Delta}},\mathbf{\ddot{\Delta}},\mathbf{\dddot{\Delta}},\mathbf{\mathcal{L}}$
satisfy similar estimates as $\dot{\Delta},\ddot{\Delta},\dddot{\Delta},L$
respectively. Indeed, for $|\beta|=1$ we can derive this equation
by direct computation. For general $\beta,$ we can derive this by
induction on $|\beta|$. This equation is of the same type as the
equation (\ref{eq:zs=00003DA0z+...for higher derivative estimates}).
Copying the proofs of Lemma \ref{lem:g''>cg-ce^=00007B-cs=00007D},
Lemma \ref{lem:g(s)<max=00007Bg(s0),g(s*)=00007D} and Lemma \ref{lem: <z(s),e>-<z(s0),e>},
we can get the desired estimate for $\mathcal{W}$. In particular,
we get the estimates for $(Q_{0}z)_{s}$ and $A_{0}Q_{0}z.$

From the equation $z_{t}=M_{0}A_{0}Q_{0}z+M_{0}Q_{0}S_{0}z_{out}$
we get the estimate for $z_{t}$. Applying $P_{0}$ to the equation
(\ref{eq:zs=00003DA0z+...for higher derivative estimates}), we get
\[
(P_{0}z)_{s}=P_{0}\dot{\Delta}z_{t}+P_{0}\ddot{\Delta}z_{out}+P_{0}\dddot{\Delta}z-P_{0}L.
\]
\textcolor{red}{{} }This equation together with the estimate of $\dddot{\Delta}z$
(See formula \ref{eq:bar=00007BA=00007D-A}) gives us the desired
estimate for $P_{0}z_{s}.$ Then the estimate for $z_{s}$ follows
from $z_{s}=P_{0}z_{s}+Q_{0}z_{s}$.
\end{proof}

\begin{lem}
\label{lem:exponential estimate for a and theta}Let $\vartheta_{0}=\int_{0}^{1}\left[\vartheta\left(\frac{s_{0}+\mathfrak{s}}{2},t\right)-Tt\right]dt,$
$a_{0}=\int_{0}^{1}\left[a\left(\frac{s_{0}+\mathfrak{s}}{2},t\right)-Ts_{0}\right]dt$,
$\tilde{a}=a(s,t)-Ts-a_{0}$ and $\tilde{\vartheta}=\vartheta(s,t)-Tt-\vartheta_{0}$.
Under the assumption of Lemma \ref{lem:higher derivative estimate},
we have for $s\in[s_{0},\mathfrak{s}]$ and all multi index $\beta$

\begin{align*}
 & \Vert\partial^{\beta}\left(\tilde{a}(s,t\right)\Vert^{2},\left\Vert \partial^{\beta}\left(\tilde{\vartheta}(s,t)\right)\right\Vert ^{2}\\
\leqq & C_{1}\underset{|\beta'|\leqq|\beta|+3}{\max}\{\Vert Q_{0}\partial^{\beta'}z(s_{0})\Vert_{0}^{2},\Vert Q_{0}\partial^{\beta'}z(\mathfrak{s})\Vert_{0}^{2}\}\\
 & +C_{1}\max\left\{ \left\Vert \tilde{a}(s_{0},\cdot)\right\Vert ^{2}+\left\Vert \tilde{\vartheta}(s_{0},\cdot)\right\Vert ^{2},\left\Vert \tilde{a}(\mathfrak{s},\cdot)\right\Vert ^{2}+\left\Vert \tilde{\vartheta}(\mathfrak{s},\cdot)\right\Vert ^{2}\right\} +o(c_{2}),
\end{align*}
where the norm $\left\Vert \cdot\right\Vert $ is $L^{2}$-norm, $o(c_{2})$
satisfies $\underset{c_{2}\to0}{\lim}o(c_{2})=0,$ and $C_{1}$ is
a constant independent of $\tilde{u}.$\end{lem}
\begin{proof}
We can modify the proofs of Lemmata 3.8-3.13 in \cite{Finite energy cylinders of small area}
in the obvious way similar to what we did in the proof of Lemma \ref{lem:g''>cg-ce^=00007B-cs=00007D}
and then use Lemma \ref{lem:higher derivative estimate} to prove
this lemma. We omit the proof here, since essentially it is not new.
\footnote{The proof of Proposition 3.4 in \cite{morse bott} is inaccurate,
and this lemma fills in the gap.%
}\end{proof}
\begin{rem}
When $\mathfrak{s}$ is infinity, we can get a better exponential
decay estimate using the same proof, and in that case the term $o(c_{2})$
can be replaced by $c_{2}e^{-(s-s_{0})}$. 
\end{rem}

Now we are ready to prove Theorem \ref{thm:converge to reeb orbit}.
\begin{proof}
Let's follow the proof in \cite{morse bott}. By Theorem \ref{thm:subsequence convergence to Reeb},
we can find a sequence $s_{0m}\to\infty$ such that 

\[
\underset{m\to\infty}{\lim}u(s_{0m},t)=\gamma(Tt)
\]
\[
\underset{m\to\infty}{\lim}a(s_{0m},t)=\pm\infty
\]
for some $T$-periodic orbit $\gamma$ of $\mathbf{R}_{\infty}$.
From the proof of Theorem \ref{thm:subsequence convergence to Reeb},
we can further require for any multi-indices $\alpha$ with $|\alpha|>0$
we have $\underset{t\in S^{1}}{\sup}\Vert\partial^{\alpha}z(s_{0m},t)\Vert\to0$
as $m\to+\infty$.

Given $\sigma>0,$ let $\zeta_{m}>0$ be the largest number such that
$u(s,t)\in S^{1}\times[-\sigma,\sigma]^{2n}$ for all $s\in[s_{0m},s_{0m}+\zeta_{m}]$.
Let $\theta_{m}:=[s_{0m},s_{0m}+\zeta_{m}]\times S^{1}$ and $\kappa_{0m}(s):=(\vartheta(s_{0m},0)-\vartheta(s,0),z_{in}(s_{0m},0)-z_{in}(s,0)).$
Now we can define the operator $A_{0m}$ similar as before in the
obvious way.

By Corollary \ref{coro:a(s,t)-Ts  and zout converges to 0.}, given
$\delta>0$ we have 
\[
\underset{(s,t)\in\theta_{m}}{\sup}\left|\partial^{\beta}(a(s,t)-Ts)\right|\leqq\delta
\]
for those multi-indices $\beta$ with $|\beta|>0,$ when $m$ is large.
This implies $a(s_{0m},0)\to+\infty,$ as $m\to+\infty$. Notice that
the other requirements in the Lemma \ref{lem:g''>cg-ce^=00007B-cs=00007D}
and Lemma \ref{lem:higher derivative estimate} are also satisfied,
i.e. given $\delta>0$, there exists $m_{\delta}$ such that for $m>m_{\delta}$
we have 

\[
\underset{(s,t)\in\theta_{m}}{\sup}\left|\partial^{\beta}z_{out}(s,t)\right|\leqq\delta
\]
 for multi-indices $\beta$, and 

\begin{equation}
\underset{(s,t)\in\theta_{m}}{\sup}\left|\partial^{\beta}(\vartheta(s,t)-Tt)\right|\leqq\delta\label{eq:derivative of (theta(s,t)-t) converges to 0}
\end{equation}
\[
\underset{(s,t)\in\theta_{m}}{\sup}\left|\partial^{\beta}z_{in}(s,t)\right|\leqq\delta
\]
 for those multi-indices $\beta$ with $|\beta|>0$. Indeed, if $\{(s_{m_{k}},t_{m_{k}})\}$
violates one of these properties, we can define $\tilde{u}_{m_{k}}(s,t)$
to be $(a(s-s_{m_{k}},t-t_{m_{k}})-a(s_{m_{k}},t_{m_{k}}),u(s-s_{m_{k}},t-t_{m_{k}}))$.
By Ascoli-Arzela, we can extract a subsequence, still called $\tilde{u}_{m_{k}}(s,t),$
such that $\tilde{u}_{m_{k}}(s,t)$ converges in $C_{loc}^{\infty}$
to a $J_{\infty}$-holomorphic cylinder $\tilde{u}_{\infty}$ over
a periodic orbit $\gamma'\in N.$ Since $\tilde{u}_{\infty}$ must
satisfy those three properties, we get a contradiction.

By construction $|\left\langle z(s_{0m}),e\right\rangle _{0m}|\to0$
and $\Vert Q_{0m}\partial^{\alpha}z(s_{0m})\Vert\to0,$ for all multi-index
$\alpha$ with $|\alpha|\geqq0.$ Let $\bar{\kappa}_{m}$ be the ``$\bar{\kappa}$''
in Lemma \ref{lem:g''>cg-ce^=00007B-cs=00007D} and Lemma \ref{lem:higher derivative estimate}
applied to $\tilde{u}|_{\theta_{m}}$ and let $\mathfrak{s}_{m}:=\sup\left\{ s\in[s_{0m},s_{0m}+\zeta_{m}]:|\kappa_{0m}(s')|\leqq\bar{\kappa}_{m}\mbox{ for all }s'\in[s_{0},s]\right\} $,
and notice that actually $\bar{\kappa}_{m}$ can be chosen independent
of $m$. We can extract a subsequence so that $u(\mathfrak{s}_{m},t)$
converges to a closed Reeb orbit $\gamma''\in N$. Therefore, $\Vert Q_{0m}\partial^{\alpha}z(\mathfrak{s}_{m})\Vert\to0$,
for all multi-indices $\alpha$ with $|\alpha|\geqq0$. Since $\left\langle z(\mathfrak{s}_{m}),e\right\rangle _{0}\to0$
and $\underset{t\in S^{1}}{\sup}\left|\frac{\partial}{\partial t}z_{in}(\mathfrak{s}_{m},t)\right|\to0,$
we obtain $\underset{t\in S^{1}}{\sup}\left|z_{in}(\mathfrak{s}_{m},t)\right|\to0$.
By Lemma \ref{lem:g''>cg-ce^=00007B-cs=00007D} and Lemma \ref{lem:higher derivative estimate},
we have 

\begin{equation}
\underset{s\in[s_{0m},\mathfrak{s}_{m}]}{\sup}\Vert\partial^{\beta}z(s)\Vert_{0m}\to0\label{eq:W1,k norm of z-l}
\end{equation}
 for $|\beta|\leqq k$. Therefore, 
\begin{align*}
 & \underset{(s,t)\in[s_{0m},\mathfrak{s}_{m}]\times S^{1}}{\sup}|z_{in}(s,t)|\\
\leqq & \underset{s\in[s_{0m},\mathfrak{s}_{m}]}{\sup}\Vert z_{in}(s,\cdot)\Vert_{C^{0}(S^{1})}\\
\leqq & C\underset{s\in[s_{0m},\mathfrak{s}_{m}]}{\sup}\Vert z_{in}(s,\cdot)\Vert_{W^{1,2}(S^{1})}\\
\leqq & C_{1}\left\{ \underset{s\in[s_{0m},\mathfrak{s}_{m}]}{\sup}\left\Vert \frac{\partial}{\partial t}z_{in}(s,\cdot)\right\Vert _{0l}+\underset{s\in[s_{0m},\mathfrak{s}_{m}]}{\sup}\Vert z_{in}(s,\cdot)\Vert_{0m}\right\} \\
\to & 0.
\end{align*}
Lemma \ref{lem:exponential estimate for a and theta} and formula
(\ref{eq:derivative of (theta(s,t)-t) converges to 0}) imply $|\vartheta(\mathcal{\mathfrak{s}}_{m},0)-\vartheta(s_{0m},0)|\to0,$
as $m\to\infty.$ Thus, we have $\mathfrak{s}_{m}=s_{0m}+\zeta_{m}$
for $m$ large enough, and 
\[
\underset{(s,t)\in[s_{0m},s_{0m}+\zeta_{m}]\times S^{1}}{\sup}|z(s,t)|\to0
\]
 as $m\to\infty$. Therefore, $\zeta_{m}=+\infty$ for $m$ large.
\end{proof}
Furthermore, we can show that the convergence of $J$-holomorphic
curve is exponentially fast. Let's prove Theorem \ref{thm: exponential convergence}.
\begin{proof}
Now with the help of the previous lemmata, the proof of the third
inequality is almost evident. Indeed, since $\mathfrak{s}=+\infty$,
Lemma \ref{lem:g(s)<max=00007Bg(s0),g(s*)=00007D} becomes $g_{0}(s)\leqq\left(g_{0}(s_{0})+\frac{c_{2}}{c_{1}^{2}-c^{2}}\right)e^{-c(s-s_{0})}.$
Consequently, in the proof Lemma \ref{lem: <z(s),e>-<z(s0),e>}, we
can get 
\begin{align*}
\left|\left\langle z(s),e\right\rangle _{0}\right| & \leqq\int_{s}^{+\infty}\left[2C\Vert Q_{0}z(\mathfrak{x})\Vert_{0}+c_{2}e^{-c_{1}(\mathfrak{x}-s_{0})}\right]d\mathfrak{x}\leqq C'e^{-c(s-s_{0})},
\end{align*}
 where $C'$ is independent of $s.$ Similarly, we can get the corresponding
statement of Lemma \ref{lem:higher derivative estimate} for $\mathfrak{s}=+\infty.$ 

The proof for the rest is a straightforward modification of the original
proof in \cite{Finite energy plane}.
\end{proof}
So far we studied the behaviors of a finite energy $J$-holomorphic
curve whose domain is an infinite cylinder. In order to compactify
the moduli space of holomorphic curves, we also need to understand
the behavior of a finite energy $J$-holomorphic curve whose domain
is a long but finite interval and whose $\omega$ energy is small.
To do that, we need the following Bubbling Lemma.
\begin{lem}
\label{lem:Bubbling-Lemma}(Bubbling Lemma \cite{compactness,Hofer viterbo})
Let $J^{0}$ be a cylindrical almost complex structure on $W=\mathbb{R}^{+}\times V$.
There exists a constant $\hbar>0$ depending only on $(W,J^{0},\omega^{0})$
where $J^{0}=J_{\infty}^{0}$ and $\omega^{0}=\omega_{\infty}^{0}$
(See Definition \ref{def: asympt cylindrical}, \ref{def:cylindrical almost complex}
and Section \ref{sub:Energy-of--holomorphic}), so that the following
holds true. Let $(J^{n},\omega_{\infty}^{n})$ be a sequence of pairs
satisfying (AC1)-(AC5) on $W$ and converging to $(J^{0},\omega^{0})$
in $C_{loc}^{\infty}$-sense. Consider a sequence of $J^{n}-$holomorphic
maps $\tilde{u}_{n}=(a_{n},u_{n})$ from the unit disc $B(0,1)$ to
$W$ satisfying $E_{n}(\tilde{u}_{n})=E_{\omega^{n}}(\tilde{u}_{n})+E_{\lambda_{n}}(\tilde{u}_{n})\leqq C$
(See Section \ref{sub:Energy-of--holomorphic}) for some constant
$C$, such that the sequence $a_{n}(0)$ are bounded, and such that
$\left\Vert \nabla\tilde{u}_{n}(0)\right\Vert \to+\infty$ as $n\to+\infty$.
Then there exist a sequence of points $z_{n}\in B(0,1)$ converging
to $0,$ sequences of positive numbers $\varepsilon_{n}$ and $R_{n}$
satisfying

\[
\begin{array}{cc}
\varepsilon_{n}\to0,\qquad & R_{n}\to+\infty,\\
\varepsilon_{n}R_{n}\to+\infty,\qquad & |z_{n}|+\varepsilon_{n}<1,
\end{array}
\]
 such that the rescaled maps 

\[
\tilde{u}_{n}^{0}:B(0,\varepsilon_{n}R_{n})\to W
\]
\[
z\mapsto\tilde{u}_{n}(z_{n}+R_{n}^{-1}z)
\]
 converge in $C_{loc}^{1}$ to a $J_{0}$-holomorphic map $\tilde{u}^{0}:\mathbb{C}\to W$
which satisfies $E(\tilde{u}^{0})\leqq C$ and $E_{\omega^{0}}(\tilde{u}^{0})>\hbar.$

Moreover, this map is either a $J_{0}$-holomorphic plane asymptotic
as $|z|\to\infty$ to a periodic orbit of the vector field $\mathbf{R}^{0}$
defined by $\mathbf{R}^{0}=J_{0}\left(\frac{\partial}{\partial r}\right),$
or extendable to a $J_{0}$-holomorphic sphere $\mathbb{P}^{1}\to W$
by Gromov's Removable of Singularity theorem.

Similar statement is also true for $\mathbb{R}^{-}\times V.$\end{lem}
\begin{proof}
See \cite{Hofer viterbo}.
\end{proof}
The following theorem studies the behavior of a long cylinder having
small $\omega$-area. It is needed in order to prove the compactness
results for the moduli space of $J$-holomorphic curves in the Symplectic
Field Theory. Refer to \cite{Finite energy cylinders of small area,compactness}
for the cylindrical case.
\begin{thm}
\label{thm: holomorphic cylinder with small area lies}Suppose that
$J$ is an asymptotically cylindrical almost complex structure on
$W=\mathbb{R}^{\pm}\times V$ at $\pm\infty$. Suppose that $J$ is
of the Morse-Bott type. Given $E_{0}>0$ and $\varepsilon>0$, there
exist constants $\sigma,c>0$ such that for every $R>c$ and every
$J$-holomorphic cylinder $\tilde{u}=(a,u):[-R,R]\times S^{1}\to W$
satisfying the inequalities $E_{\omega}(\tilde{u})<\sigma$ and $E(\tilde{u})<E_{0},$
we have $u(s,t)\in B_{\varepsilon}(u(0,t))$ for all $s\in[-R+c,R-c]$
and all $t\in S^{1}$. \end{thm}
\begin{proof}
The proof follows the scheme in \cite{compactness} with some modification.

By contradiction, assume that there exist sequences $c_{n}\to+\infty,$
$R_{n}>c_{n}$ and $\tilde{u}_{n}=(a_{n},u_{n}):[-R_{n},R_{n}]\times S^{1}\to W$.
$\tilde{u}_{n}$ is $J-$holomorphic satisfying $E(\tilde{u}_{n})\leqq E_{0}$,
$E_{\omega}(\tilde{u}_{n})\to0$, and $u_{n}(s_{n},t_{n})\notin B(u_{n}(0,t_{n}),\epsilon)$
for some $s_{n}\in[-k_{n},k_{n}]$, $k_{n}=R_{n}-c_{n}$ and $t_{n}\in S^{1}.$
By the proof of Proposition \ref{pro:gradient bound for finite energy curve}
together with the Bubbling Lemma, $\Vert\nabla\tilde{u}_{n}\Vert$
is uniformly bounded on each compact subset. We can extract a subsequence
of $n$, still denoted by $n$, such that $a_{n}(s_{n},t_{n})\to\pm\infty.$
This is because otherwise, we can get a contradiction as in the proof
of Proposition \ref{pro:gradient bound for finite energy curve}.
Now define $\tilde{u}_{n}^{0}(s,t):=(a_{n}^{0},u_{n}^{0})=(a_{n}(s,t)-a_{n}(s_{n},t_{n}),u_{n}(s,t)).$
Hence, by Ascoli-Arzela, we can extract a subsequence still called
$\tilde{u}_{n}^{0}$ converging to a $J_{\infty}$-holomorphic cylinder
$\tilde{u}:\mathbb{R}\times S^{1}\to\mathbb{R}\times V.$ Since $\tilde{u}$
satisfies $E_{\omega}(\tilde{u})=0$ and $E(\tilde{u})\leqq E_{0}$,
$\tilde{u}$ is a trivial cylinder over some periodic orbit $\gamma.$
Let's choose a neighborhood around $\gamma$ and pick the coordinate
as in Lemma \ref{lem:choosing nice coordinate}, and show that
\begin{equation}
\underset{(s,t)\in[-k_{n},k_{n}]\times S^{1}}{\sup}|\partial^{\beta}z_{out,n}(s,t)|\to0\label{eq:zout,n}
\end{equation}
 for multi-indices $\beta$ and
\begin{equation}
\underset{(s,t)\in[-k_{n},k_{n}]\times S^{1}}{\sup}|\partial^{\beta}(a_{n}(s,t)-Ts)|\to0\label{eq:an(s,t)}
\end{equation}
\begin{equation}
\underset{(s,t)\in[-k_{n},k_{n}]\times S^{1}}{\sup}|\partial^{\beta}z_{in,n}(s,t)|\to0\label{eq:zin,n(s,t)}
\end{equation}
\begin{equation}
\underset{(s,t)\in[-k_{n},k_{n}]\times S^{1}}{\sup}|\partial^{\beta}(\vartheta_{n}(s,t)-Tt)|\to0\label{eq:vartheta n (s,t)}
\end{equation}
 for multi-indices $\beta$ with $|\beta|>0$, when $n\to+\infty.$ 

If this were not true, suppose there exists a subsequence of $\{n\}$
still denoted by $\{n\}$ such that $(s'_{n},t'_{n})$ violates one
of these properties. Then we can do the same argument using $(s'_{n},t'_{n})$
instead of $(s_{n},t_{n})$ as above, and get a trivial cylinder contradicting
the fact that $(s'_{n},t'_{n})$ violates one of these properties. 

Define $A_{0n}$ and $Q_{0n}$ in the obvious way using $\gamma$
and $s_{0n}=0$. Then we can apply Lemma \ref{lem:g''>cg-ce^=00007B-cs=00007D},
Lemma \ref{lem:g(s)<max=00007Bg(s0),g(s*)=00007D}, Lemma \ref{lem: <z(s),e>-<z(s0),e>},
and Lemma \ref{lem:higher derivative estimate} to each $\tilde{u}_{n}|_{[-k_{n},k_{n}]}$,
and get $\underset{s\in[-k_{n},k_{n}]}{\sup}\Vert Q_{0n}z_{n}(s)\Vert_{0,n}\to0.$
Then the Sobolev embedding theorem tells us $\kappa_{0n}\to0$ as
$n\to+\infty.$ This contradicts the assumption that $u_{n}(s_{n},t_{n})\notin B(u_{n}(0,t),\epsilon)$.
\end{proof}
We need the following theorem later to prove the surjectivity of the
gluing map in the subsequent paper. After we proved all the previous
lemmata and theorems, the proof of the following theorem is standard.
For the case when $J$ is cylindrical and non-degenerate, and $V$
is a contact manifold, the proof is given in \cite{Finite energy cylinders of small area}.
\begin{thm}
\label{thm:exponential estimate of cylinder with small area}\textsl{Suppose
that $J$ is an asymptotically cylindrical almost complex structure
on $W=\mathbb{R}^{+}\times V$ at $\infty$. Suppose that }$J$ is
of the Morse-Bott type\textsl{. Given $E_{0}>0$ and sufficiently
small $\varepsilon>0$, there exist constants $\sigma,c,\flat,\nu>0$
such that for every $R>c$ and every $J$-holomorphic cylinder $\tilde{u}=(a,u):[-R,R]\times S^{1}\to(\flat,\infty)\times V$
satisfying the inequalities $E_{\omega}(\tilde{u})<\sigma$ and $E(\tilde{u})<E_{0},$
there exists either a point $w\in W$ such that $\tilde{u}(s,t)\in B_{\varepsilon}(w)$
for $s\in[-R+c,R-c]$ and $t\in S^{1},$ or a $T$-periodic orbit
$\gamma$ of $\mathbf{R}_{\infty}$ such that $u(s,t)\in B_{\varepsilon}(\gamma(Tt))$
for $s\in[-R+c,R-c]$ and $t\in S^{1}.$ In the second case, we have
a coordinate around $\gamma$ as in Lemma \ref{lem:choosing nice coordinate}
such that }

\textsl{
\begin{align*}
|D^{\beta}\{a(s,t)-Ts-a_{0}\}|^{2} & \leqq\varepsilon^{2}M_{\beta}\frac{\cosh(2\nu s)}{\cosh(2\nu(R-c))}+C_{\beta}e^{-c_{\beta}(s+R-c)},\\
|D^{\beta}\{\vartheta(s,t)-Tt-\vartheta_{0}\}|^{2} & \leqq\varepsilon^{2}M_{\beta}\frac{\cosh(2\nu s)}{\cosh(2\nu(R-c))}+C_{\beta}e^{-c_{\beta}(s+R-c)},\\
|D^{\beta}z(s,t)|^{2} & \leqq\varepsilon^{2}M_{\beta}\frac{\cosh(2\nu s)}{\cosh(2\nu(R-c))}+C_{\beta}e^{-c_{\beta}(s+R-c)},
\end{align*}
 for $s\in[-R+c,R-c],$ $t\in S^{1},$ and $\beta\in\mathbb{N}\times\mathbb{N}$
such that $|\beta|\leqq l-3$, where $M_{\beta},C_{\beta},c_{\beta}$
are constants independent of $\tilde{u}$ and $\varepsilon,$ $C_{\beta}$
converges to $0$ as $\flat$ converges to $+\infty,$ and $M_{\beta}$
and }\textup{$c_{\beta}$}\textsl{ are independent of $\flat.$}

Similar statement is also true for $\mathbb{R}^{-}\times V.$
\end{thm}

\section{\label{sec:Almost-complex-manifolds with ends}Almost complex manifolds
with asymptotically cylindrical ends}

In this section, we introduce the notion of almost complex manifolds
with asymptotically cylindrical ends.

\subsection{\label{sub:Definition 4.1}Definition}

Let $(W_{0},\omega')$ be a closed symplectic manifold with boundary
$\partial W_{0}=V_{+}\bigsqcup V_{-},$ where $V_{\pm}$ is an oriented
closed manifold. Let $W$ be the noncompact smooth manifold obtained
by attaching $E_{\pm}:=\mathbb{R}^{\pm}\times V_{\pm}$ to $W_{0}$
along $\{0\}\times V_{\pm}$ and $V_{\pm}.$ Suppose that there exists
an almost complex structure $J$ on $W$ such that $J|_{W_{0}}$ is
compatible with $\omega'$ and $(E_{\pm},J|_{E_{\pm}})$ is asymptotically
cylindrical at $\pm\infty.$ We assume that the orientation of $E_{\pm}$
determined by $J|_{E_{\pm}}$ coincides with the orientation coming
from the standard orientation of $\mathbb{R}^{\pm}$ and the orientation
of $V_{\pm}.$ This assumption distinguishes $V_{+}$ from $V_{-}.$
Further more, we assume $\omega'|_{V_{\pm}}=\omega_{\pm\infty},$
where $\omega_{\pm\infty}$ is the $2$-form on $V_{\pm}$ from Definition
\ref{def: asympt cylindrical}. In this case, we say $(W,J)$ is an
almost complex manifold with asymptotically cylindrical ends.

\begin{example}
\label{exm:sub kahler}\cite{compactness} Let $(X,\omega',J)$ be
an almost Kähler manifold, and $Y\subset X$ is an embedded closed
almost Kähler submanifold. We claim that $(X\backslash Y,J|_{X\backslash Y})$
has asymptotically cylindrical negative end. Let $N$ be the normal
bundle of $Y$ in $X$ with the metric $\omega'(\cdot,J\cdot)|_{Y}$,
$V$ be the associated unit sphere bundle of $N$ defined by $V=\left\{ u\in N:\left|u\right|=1\right\} $,
and $U_{\epsilon}$ be the disc bundle of $N$ defined by $U_{\epsilon}=\left\{ u\in N:\left|u\right|\leqq\epsilon\right\} $.
For small enough $\epsilon>0$, $U_{\epsilon}$ is diffeomorphic to
a tubular neighborhood of $Y$ in $X$ via the exponential map with
respect to the metric $\omega'(\cdot,J\cdot)$.  Since $U_{\epsilon}\backslash Y$
is also diffeomorphic to $(-\infty,\log\epsilon]\times V$ via the
map $u\mapsto(\log|u|,u/|u|)$, one can check that this makes $(X\backslash Y,J|_{X\backslash Y})$
an almost complex manifold with asymptotically cylindrical negative
end. 

In particular, if we pick $Y$ to be a point in $X$, we get Example
\ref{exm: R(2n+2)} as a special case.  
\end{example}

\subsection{Energy of $J$-holomorphic curves}

Let $w$ be a $J$-holomorphic map from a punctured Riemann surface
$(\Sigma,j)$ to $(W,J)$, and define 

\[
E_{\mbox{symp}}(w)=\int_{w^{-1}(W_{0})}w^{*}\omega',
\]

\[
E_{\omega}(w)=\int_{w^{-1}(E_{+})}w^{*}\omega+\int_{w^{-1}(E_{-})}w^{*}\omega,
\]

\[
E_{\lambda}(w)=\underset{\phi\in\mathcal{C}_{+}}{\sup}\int_{w^{-1}(E_{+})}w^{*}(\phi\sigma\wedge\lambda)+\underset{\phi\in\mathcal{C}_{-}}{\sup}\int_{w^{-1}(E_{-})}w^{*}(\phi\sigma\wedge\lambda),
\]
 where 
\[
\mathcal{C}_{+}=\{\phi\in C_{c}^{\infty}(\mathbb{R}^{+},[0,1])|\intop\phi=1\}
\]
 
\[
\mathcal{C}_{-}=\{\phi\in C_{c}^{\infty}(\mathbb{R}^{-},[0,1])|\int\phi=1\},
\]
 and 
\[
E(w)=E_{\mbox{symp}}(w)+E_{\omega}(w)+E_{\lambda}(w).
\]

\begin{thm}
\label{thm: asymptotically cylindrical ends}Suppose $(W,J)$ is an
almost complex manifold with asymptotically cylindrical ends. Suppose
that $J$ is of the Morse-Bott type. Let $w$ be a $J$-holomorphic
curve from a puncture Riemann surface $\Sigma$ to $W$ with $E(w)<\infty.$
Then around each puncture, either $w$ can be extended holomorphically
over the puncture, or one can choose a holomorphic coordinate chart
$\mathbb{R}^{+}\times S^{1}$ or $\mathbb{R}^{-}\times S^{1}$ in
$S$ around the puncture, such that $w$ converges to a Reeb orbit
in $E_{+}$ or $E_{-}$ in the sense of Theorem \ref{thm:converge to reeb orbit}
and \ref{thm: exponential convergence}.\end{thm}
\begin{proof}
If $w$ is bounded around a puncture, then Gromov's Removable of Singularity
Theorem implies that $w$ can be extended holomorphically over the
puncture. 

Suppose that $w$ is not bounded around a puncture. We pick a holomorphic
cylindrical coordinate $\mathbb{R}^{+}\times S^{1}$ around the puncture
of $\Sigma.$ By Proposition \ref{pro:gradient bound for v holomorphic cylinder},
$|\nabla w|<C$ with respect to the standard metric on $\mathbb{R}^{+}\times S^{1}.$
If $w$ keeps coming back to a compact region of $W$ and also escaping
to the positive (or negative) end of $W,$ we can find $r_{0}$ such
that $w$ touches $\{r_{0}\}\times V_{\pm}$ and $\{r_{0}\pm3C\}\times V_{\pm}$
infinitely many times. Then we can apply Gromov's Monotonicity Theorem
to $w$ in the region $[r_{0}\pm C,r_{0}\pm2C]\times V_{\pm}$ as
in the argument of Case 1 in the proof of Theorem \ref{thm:subsequence convergence to Reeb},
and get $E(w)=\infty,$ which contradicts to the assumption. Therefore,
near the puncture $w$ converges to $\infty$ or $-\infty$ in $E_{+}$
or $E_{-}.$ Then Theorem \ref{thm: asymptotically cylindrical ends}
follows from Theorem \ref{thm:converge to reeb orbit} and \ref{thm: exponential convergence}.\end{proof}
\begin{prop}
\label{prop: topological energy bound}Suppose $(W,J)$ is an almost
complex manifold with asymptotically cylindrical ends. Suppose that
$J$ is of the Morse-Bott type. Then there exists a constant $\epsilon_{0}>0$
such that if $K_{0}^{\pm}<\epsilon_{0},$ where $K_{0}^{\pm}$ is
the constant in (AC1), the following holds.

Let $w$ be a $J$-holomorphic curve from a puncture Riemann surface
$\Sigma$ to $W$ such that around punctures of $\Sigma$, $w$ converges
to the periodic orbits $\gamma_{1}^{+},...,\gamma_{p}^{+}$ inside
$V_{+}$ and $\gamma_{1}^{-},...,\gamma_{q}^{-}$ inside $V_{-}.$
Then we have
\begin{align*}
E(w)\leq & C_{1}\sum_{i=1}^{p}\int_{\gamma_{i}^{+}}\lambda_{\infty}-C_{2}\sum_{j=1}^{q}\int_{\gamma_{j}^{-}}\lambda_{-\infty}\\
 & +C_{3}\left(\int_{w^{-1}(E_{+})}w^{*}\omega_{\infty}+\int_{w^{-1}(W_{0})}w^{*}\omega'+\int_{w^{-1}(E_{-})}w^{*}\omega_{-\infty}\right),
\end{align*}
 where $C_{1},C_{2},C_{3}$ are some positive constants that are independent
of $w.$ In particular, $E(w)$ only depends on the homology class
of $w$ in $H_{2}(W,(\cup_{i=1}^{p}\gamma_{i}^{+})\cup(\cup_{j=1}^{q}\gamma_{j}^{-})).$
\end{prop}
Proposition \ref{prop: topological energy bound} is the asymptotically
cylindrical version of Proposition 6.13 in \cite{compactness}. The
extra work to prove it for the asymptotically cylindrical case is
essentially carried out in the Appendix of \cite{Bao} where we assume
$\omega_{\pm\infty}=d\lambda_{\infty}$. For the sake of completeness,
we reproduce the proof here.
\begin{proof}
First, we restrict ourself to $E_{+},$ and denote $w_{\pm}:=w|_{w^{-1}(E_{\pm})}.$
Note that when restricted to $J$-complex planes, we have 
\begin{equation}
|\omega-\omega_{\infty}|\leq\epsilon e^{-\delta s}(\omega+\sigma\wedge\lambda),\label{eq:omega-omegainfty}
\end{equation}
\begin{equation}
|d\lambda_{\infty}|\leq C\omega+\epsilon e^{-\delta s}\sigma\wedge\lambda,\label{eq:fig}
\end{equation}
 and
\begin{equation}
|\sigma\wedge\lambda-dr\wedge\lambda_{\infty}|\leq\epsilon e^{-\delta s}(\sigma\wedge\lambda+\omega),\label{eq:palm}
\end{equation}
where $C$ is a positive constant and the constant $\epsilon>0$ can
be chosen to be small if $K_{0}^{+}$ is small. Since $\int_{0}^{\infty}\delta e^{-\delta s}ds=1,$
we get 
\[
\int_{w^{-1}(E_{+})}w^{*}\omega\leq\int_{w^{-1}(E_{+})}w^{*}\omega_{\infty}+\epsilon\int_{w^{-1}(E_{+})}w^{*}\omega+\frac{\epsilon}{\delta}E_{\lambda}(w_{+}),
\]
 where $E_{\lambda}(w_{\pm}):=\underset{\phi\in\mathcal{C}_{\pm}}{\sup}\int_{w^{-1}(E_{\pm})}w^{*}(\phi\sigma\wedge\lambda).$
Absorbing the second term on the RHS to the LHS, we get 
\begin{equation}
E_{\omega}(w_{+})\leq C_{1}\int_{w^{-1}(E_{+})}w^{*}\omega_{\infty}+C_{2}\epsilon E_{\lambda}(w_{+}),\label{eq:apple}
\end{equation}
for some constants $C_{1},C_{2},$ where $E_{\omega}(w_{\pm}):=\int_{w^{-1}(E_{\pm})}w^{*}\omega.$

For any $\phi\in\mathcal{C}_{+},$ let $\Phi(s)=\int_{0}^{s}\phi(l)dl,$
and then using \eqref{eq:fig} and \eqref{eq:palm} we have 
\begin{eqnarray*}
 &  & \int_{w^{-1}(E_{+})}w^{*}\phi\sigma\wedge\lambda\\
 & = & \int_{w^{-1}(E_{+})}w^{*}\phi dr\wedge\lambda_{\infty}+\int_{w^{-1}(E_{+})}w^{*}\phi(\sigma\wedge\lambda-dr\wedge d\lambda_{\infty})\\
 & \leq & \int_{w^{-1}(E_{+})}w^{*}d(\Phi\lambda_{\infty})-\int_{w^{-1}(E_{+})}w^{*}\Phi d\lambda_{\infty}\\
 &  & +\int_{w^{-1}(E_{+})}w^{*}\epsilon e^{-\delta s}\phi(\sigma\wedge\lambda+\omega)\\
 & \leq & \sum_{i=1}^{p}\int_{\gamma_{i}^{+}}\lambda_{\infty}+\int_{w^{-1}(E_{+})}w^{*}(C\omega+\epsilon e^{-\delta s}\sigma\wedge\lambda)\\
 &  & +\int_{w^{-1}(E_{+})}w^{*}\epsilon e^{-\delta s}\phi(\sigma\wedge\lambda+\omega)\\
 & \leq & \sum_{i=1}^{p}\int_{\gamma_{i}^{+}}\lambda_{\infty}+CE_{\omega}(w_{+})+\epsilon E_{\lambda}(w_{+}),
\end{eqnarray*}
where in the last inequality we get the constants $C$ and $\epsilon$
by slightly abusing the notations, but we can still have $\epsilon$
small. Taking sup over $\phi,$ we get 
\begin{equation}
E_{\lambda}(w_{+})\leq\sum_{i=1}^{p}\int_{\gamma_{i}^{+}}\lambda_{\infty}+CE_{\omega}(w_{+})+\epsilon E_{\lambda}(w_{+}).\label{eq:almost done}
\end{equation}
 Therefore, by (\ref{eq:apple}) and (\ref{eq:almost done}) we have
\begin{equation}
E_{\omega}(w_{+})+E_{\lambda}(w_{+})\leq C_{1}\int_{\gamma_{+}}\lambda_{\infty}+C_{2}\int_{w^{-1}(E_{+})}w^{*}\omega_{\infty},\label{eq:+}
\end{equation}
 where constants $C_{1}$ and $C_{2}$ are not necessarily the same
as before.

For $E_{-},$ by the proof of Theorem 10 in \cite{Bao} if $K_{0}^{-}$
is small, we have 
\begin{equation}
E_{\omega}(w_{-})+E_{\lambda}(w_{-})\leq C_{1}^{'}E_{symp}(w)+C_{2}^{'}\int_{w^{-1}(E_{-})}w^{*}\omega_{\infty}-C_{3}^{'}\sum_{j=1}^{q}\int_{\gamma_{j}^{-}}\lambda_{-\infty},\label{eq:-}
\end{equation}
where $C_{1}^{'},C_{2}^{'},C_{3}^{'}$ are positive constants independent
of $w.$ Here we recall 
\begin{equation}
E_{symp}(w)=\int_{w^{-1}(W_{0})}w^{*}\omega'.\label{eq:0}
\end{equation}

Now by (\ref{eq:+}) and (\ref{eq:-}) we have 
\begin{eqnarray*}
 &  & E(w)\\
 & = & E_{\omega}(w_{+})+E_{\lambda}(w_{+})+E_{\omega}(w_{-})+E_{\lambda}(w_{+})+E_{symp}(w)\\
 & \leq & a_{1}(E_{\omega}(w_{+})+E_{\lambda}(w_{+}))+a_{2}(E_{\omega}(w_{-})+E_{\lambda}(w_{+}))+a_{3}E_{symp}(w_{0})\\
 & \leq & C_{1}\int_{\gamma_{+}}\lambda_{\infty}-C_{2}\int_{\gamma_{-}}\lambda_{-\infty}\\
 &  & +C_{3}\left(\int_{w^{-1}(E_{+})}w^{*}\omega_{\infty}+\int_{w^{-1}(W_{0})}w^{*}\omega'+\int_{w^{-1}(E_{-})}w^{*}\omega_{\infty}\right),
\end{eqnarray*}
where $a_{1},a_{2},a_{3}\geq1$ are some positive constants chosen
in a way such that the last inequality holds, for some positive constants
$C_{1},C_{2}$ and $C_{3}$. 
\end{proof}
Let $\mathcal{M}_{g,p+q}^{A}(\gamma_{1}^{+},...,\gamma_{p}^{+},\gamma_{1}^{-},...,\gamma_{q}^{-};J)$
be the moduli space of $J$-holomorphic curves of genus $g$ in $W$
that converge to periodic orbits $\gamma_{1}^{+},...,\gamma_{p}^{+}$
inside $V_{+}$ and $\gamma_{1}^{-},...,\gamma_{q}^{-}$ inside $V_{-}$
and represent the homology class $A\in H_{2}(W,(\cup_{i=1}^{p}\gamma_{i}^{+})\cup(\cup_{j=1}^{q}\gamma_{j}^{-})).$
Denote by $\overline{\mathcal{M}}_{g,p+q}^{A}(\gamma_{1}^{+},...,\gamma_{p}^{+},\gamma_{1}^{-},...,\gamma_{q}^{-};J)$
the compactification of $\mathcal{M}_{g,p+q}^{A}(\gamma_{1}^{+},...,\gamma_{p}^{+},\gamma_{1}^{-},...,\gamma_{q}^{-};J)$
by allowing stable holomorphic buildings. See 8.1 and 8.2 in \cite{compactness}
for the definition of stable holomorphic buildings in manifolds with
cylindrical ends and the topology of the moduli space of holomorphic
buildings; Finally, let us state the compactness results.
\begin{thm}
Suppose $(W,J)$ is an almost complex manifold with asymptotically
cylindrical ends. Suppose that $J$ is of the Morse-Bott type.

Then $\overline{\mathcal{M}}_{g,p+q}^{A}(\gamma_{1}^{+},...,\gamma_{p}^{+},\gamma_{1}^{-},...,\gamma_{q}^{-};J)$
is compact. \end{thm}
\begin{proof}
The extra difficulty of proof that comes from $J$ being asymptotically
cylindrical is taken care of by Theorem \ref{thm: asymptotically cylindrical ends},
the rest of the proof is a straightforward modification of \cite{compactness}.
For the sake of completeness, we outline the proof as follows.

Suppose that $(\Sigma_{n},w_{n})$ is a sequence of $J$-holomorphic
maps from punctured Riemann surface $\Sigma_{n},$ with $E(w_{n})<C.$ 

First, we add additional marked points to $\Sigma_{n}$ to stabilize
$\Sigma_{n},$ and use the unique hyperbolic metric on $\Sigma_{n}$
to decompose $\Sigma_{n}$ into $\epsilon$-thick part $\Sigma_{n}^{\epsilon-\mbox{thick}}$
and $\epsilon$-thin part $\Sigma_{n}^{\epsilon-\mbox{thin}}$ according
to the injectivity radius, for $\epsilon>0$. Take a subsequence of
$\Sigma_{n}$, still called $\Sigma_{n},$ such that $\Sigma_{n}$
converges to a nodal surface $\Sigma_{\infty}$ in the Deligne-Mumford
sense. By keep adding marked points to $\Sigma_{n},$ if necessary,
one can keep track of all the sphere bubbles of $w_{n}$ as $n\to\infty.$
Eventually, we achieved that for fixed $\epsilon>0,$ $w_{n}|_{\Sigma_{n}^{\epsilon-\mbox{thick}}}$
has uniformly gradient bound. By Ascoli-Arzela and elliptic estimates,
we can extract a converging subsequence of $w_{n},$ still called
$w_{n}.$ Now we let $\epsilon\to0,$ and pick a diagonal subsequence,
we get a converging subsequence of $w_{n},$ still called $w_{n},$
with the limit $(\Sigma_{\infty},w_{\infty}|_{\Sigma_{\infty}}).$
By Theorem \ref{thm: asymptotically cylindrical ends}, we know around
a puncture, the limit $w_{\infty}|_{\Sigma_{\infty}}$ either has
removable singularity, or converges to a Reeb orbit. But at the current
stage, $w_{\infty}$ may not be defined around the nodal points.

Secondly, for $\epsilon$ sufficiently small, the $\epsilon$-thin
part is disjoint union of finite cylinders or half finite cylinders.
If $E_{\omega}(w_{n}|_{\Sigma_{n}^{\epsilon-\mbox{thin}}})\to0$ as
$n\to\infty,$ then the behavior of the $w_{n}|_{\Sigma_{n}^{\epsilon-\mbox{thin}}}$
is controlled by Theorem \ref{thm: holomorphic cylinder with small area lies}.
In this case, the convergence of $w_{n}$ in the thick part can be
continuously extended over $\Sigma_{n}.$ Otherwise, $w_{n}|_{\Sigma_{n}^{\epsilon-\mbox{thin}}}$
can have the additional broken trajectory degeneration. By adding
more marked points to keep track of all the broken trajectory, one
has $E_{\omega}(w_{n}|_{\Sigma_{n}^{\epsilon-\mbox{thin}}})\to0$
as $n\to\infty.$\\

\end{proof}
\textbf{Acknowledgments. }Firstly, I would like to express my deepest
gratitude to my advisor Yong-Geun Oh for suggesting such an interesting
project. Secondly, I would like to thank Conan Leung for very patiently
answering my questions in Symplectic Geometry. Thirdly, I would like
to thank Lino Amorim, Garrett Alston, Dongning Wang, Rui Wang and
Ke Zhu for all the fruitful discussions and valuable suggestions.
Last but not least, I thank the anonymous referee for all the critics
and suggestions, which largely improved the quality of this paper.

\end{document}